\documentclass[10pt]{article}
\usepackage{amssymb,amsmath,amsfonts}
\usepackage{dsfont}
\usepackage{amsthm}
\usepackage{graphicx,color,subfig}
\usepackage{xspace}
\usepackage{mathrsfs}
\usepackage{amscd}
\usepackage{comment}
\numberwithin{equation}{subsubsection}
\DeclareGraphicsExtensions{.pdf, .png, .jpg, .mps}
\usepackage[colorlinks=true,pdfpagemode=FullScreen]{hyperref}
\usepackage[colorlinks=true,pdfpagemode=FullScreen]{hyperref}
\newtheorem{lemma}{Lemma}[section]
\newtheorem{theorem}[lemma]{Theorem}
\newtheorem{proposition}[lemma]{Proposition}

\theoremstyle{definition}
\newtheorem{example}[lemma]{Example}
\newtheorem{definition}[lemma]{Definition}
\newtheorem{remark}[lemma]{Remark}

\makeatletter
\def\keywords{
    \vspace{1ex}
    \noindent
    \if@twocolumn
      \small{\bf  Keywords}\/---$\!$    \else
      \begin{center}\small\ {\bf Keywords}\end{center}\quotation\small
    \fi}
\def\endkeywords{\vspace{0.6em}\par\if@twocolumn\else\endquotation\fi
    \normalsize\rm}
\makeatother

\newcommand{\calL}{\ensuremath{\mathcal L}}

\newcommand{\Z}{\ensuremath{\mathbb Z}}
\newcommand{\T}{\ensuremath{\mathbb T}}

\newcommand{\Mp}{\ensuremath{M'}}

\newcommand{\N}{{\mb{N}}}

\newcommand{\R}{{\mb{R}}}

\newcommand{\eps}{\varepsilon}

\let \Re \relax
\DeclareMathOperator{\Re}{Re}
\let \Im \relax
\DeclareMathOperator{\Im}{Im}

\newcommand{\ovl}[1]{\overline{#1}}
\newcommand{\udl}[1]{\underline{#1}}

\newcommand{\Con}{\ensuremath{\mathscr C}}

\newcommand{\Cinf}{\ensuremath{\Con^\infty}}
\newcommand{\Cinfc}{\ensuremath{\Con^\infty}_{c}}

\newcommand{\scrS}{\ensuremath{\mathscr S}}

\newcommand{\norm}[2]{| #1 |_{#2}}

\DeclareMathOperator{\supp}{supp}

\DeclareMathOperator{\Char}{Char}

\DeclareMathOperator{\diag}{diag}

\DeclareMathOperator{\essinf}{essinf}

\DeclareMathOperator{\op}{op}

\DeclareMathOperator{\Op}{Op}

\newcommand{\transp}{\ensuremath{\phantom{}^{t}}}

\renewcommand{\d}{\ensuremath{\partial}}
\newcommand{\wt}{\ensuremath{\widetilde}}

\newcommand{\nhd}{neighbourhood\xspace}

\newcommand{\hf}{\frac{1}{2}}

\newcommand{\wrt}{w.r.t.\@\xspace}

\newcommand{\ie}{i.e.\@\xspace}
\newcommand{\eg}{e.g.\@\xspace}

\DeclareMathOperator*{\ssum}{\textstyle{\sum}}

\let \div \relax
\DeclareMathOperator{\div}{div}

\def\Ut{\tilde{U}}

\newcommand{\F}{\mathscr F}

\def\p{\partial}
\def\no{\noindent}
\def\io{{\infty}}

\def\diag{\operatorname{diag}}

\def\re{\operatorname{Re}}

\def\moo{C^{\io}}

\def\N{\mathbb N}
\def\Z{\mathbb Z}

\def\R{\mathbb R}

\def\poscal#1#2{\langle#1,#2\rangle}
\def\Poscal#1#2{\left\langle#1,#2\right\rangle}

\def\poi#1#2{\left\{#1,#2\right\}}
\def\norm#1{\Vert#1\Vert}
\def\val#1{\vert#1\vert}
\def\Val#1{\left\vert#1\right\vert}
\def\valjp#1{\langle#1\rangle}

\def\l2{L^2(\R^{n})}
\def\L2{L^2(\R^{2n})}
\def\supp{\operatorname{supp}}

\def\sign{\operatorname{sign}}

\def\op#1{{\text{Op}(#1)}}

\def\RZ{\R^{2n}}

\def\sign{\operatorname{sign}}

\def\hs{{\hskip15pt}}
\def\vs{\vskip.3cm}

\let\no=\noindent

\let\no=\noindent
\def\mat22#1#2#3#4{\begin{pmatrix}#1&#2\\ #3&#4\end{pmatrix}}

\def\wt#1{\widetilde{#1}}

\def\XXint#1#2#3{{\setbox0=\hbox{$#1{#2#3}{\int}$}
     \vcenter{\hbox{$#2#3$}}\kern-.5\wd0}}

\def\beq{\begin{equation}}
\def\eeq{\end{equation}}

 \numberwithin{equation}{section}
\definecolor{ao}{rgb}{0.0, 0.5, 0.0}
\begin{document}
\title{Energy decay for a locally undamped wave equation
}

\author{
Matthieu L\'eautaud\footnote{
Institut de Math\'ematiques de Jussieu-Paris Rive Gauche, Universit\'e Paris Diderot (Paris
VII), B\^atiment Sophie Germain, 75205 Paris Cedex 13 France
\texttt{leautaud@math.univ-paris-diderot.fr}} 
\
and Nicolas Lerner\footnote{
Institut de Math\'ematiques de Jussieu-Paris Rive Gauche, Universit\'e Pierre et Marie Curie (Paris
VI), 4 Place Jussieu, 75252 Paris cedex 05, France
\texttt{nicolas.lerner@imj-prg.fr}}
\\
}
\maketitle
\begin{abstract}{
We study the decay rate for the energy of solutions of a damped wave equation in a situation where the {\it Geometric Control Condition} is violated. We assume that the set of undamped trajectories is a flat torus of positive codimension and that the metric is locally flat around this set. We further assume that the damping function enjoys locally a prescribed homogeneity near the undamped set in traversal directions.
We prove a sharp decay estimate at a polynomial rate that depends on the homogeneity of the damping function. 
Our method relies on a refined microlocal analysis linked to a second microlocalization procedure 
to cut the phase space into tiny regions respecting the uncertainty principle
but way too small to enter a standard semi-classical analysis localization.
Using a multiplier method, we obtain the energy estimates in each region 
and we then patch the microlocal estimates together. }
\end{abstract}

\section{Introduction and main results}
\subsection{Introduction}
We consider a smooth compact Riemannian manifold $(M,g)$ of dimension $n$, and denote by $\Delta_g$ the associated  negative Laplace-Beltrami operator.
Given {$b \in L^\infty(M)$}, we study the decay rates for the damped wave equation on $M$:
\begin{equation}
\label{eq: stabilization}
\left\{
\begin{array}{ll}
\d_{t}^2 u - \Delta_g u + b(x) \d_{t} u = 0 & \text{in }\R^+ \times M , \\
(u, \d_t u)|_{t=0} = (u_0,u_1)                  & \text{in } M .\\
\end{array}
\right.
\end{equation}
The energy of a solution is defined by 
\begin{align}
\label{eq: definition energy}
E(u(t)) = \frac12 (\|\nabla_g u(t)\|_{L^2(M)}^2 + \|\d_t u(t)\|_{L^2(M)}^2 ),
\end{align}
(see for instance Appendix~\ref{s:geometry} for a definition of $\Delta_g$ and $\nabla_g$) and evolves as 
$$
\frac{d}{dt}E(u(t))= - \int_M b|\d_t u|^2 dx .
$$
The energy is thus actually damped when $b\geq 0$ {a.e.} on  $M$, what we assume from now on.
{
We define the subset of $M$ on which the damping is effective as 
\begin{equation}
\label{eq:defomega}
\omega_b := \bigcup \left\{U\subset M, U \text{ open}, \essinf_U(b) >0 \right\} .
\end{equation}
}{Notice that $\omega_{b}$ is an open set included in the interior of $\supp b$ and thus $\overline{\omega}_{b}\subset\supp b$.\footnote{Remark that the converse may fail, taking for instance $b = \mathds{1}_{K}$ where $K$ is a compact Cantor set with positive measure satisfying $\mathring{K}=\emptyset$, in which case $\supp(b)=K$ and $\omega_b = \emptyset$.} In the usual case where $b$ is continuous, we have $\omega_b = \{b>0\}$ and $\overline{\omega}_{b}=\supp b$.
}{As soon as $\omega_b \neq \emptyset$ one has $E(u(t))\to 0$ as $t \to + \infty$ (see for instance
\cite{Leb:X},
\cite{Leb:96}). Moreover, a criterion for uniform (and hence exponential) decay is due to Rauch-Taylor~\cite{RT:74} (see also \cite{BLR:88} and Lemma~\ref{lem: geom control} below): there exist $C>0, \gamma>0$ such that for all data,
 $$E(u(t)) \leq C e^{-\gamma t} E(u(0)),$$ if the Geometric Control Condition (GCC) holds: every geodesic starting from $S^*M$ and traveling with unit speed (see Appendix~\ref{s:geometry} for a precise statement) enters the set $\omega_b$ in finite time. Reciprocally, if there is a geodesic that never meet $\supp(b)$, then uniform decay does not hold (see for instance~\cite{Ralston:69}). In the case $b \in \Con^0(M)$, the situation is simpler since uniform decay is {\em equivalent} to the fact that $\omega_b (= \{b>0\})$ satisfies (GCC), as remarked by Burq and G\'erard~\cite{BG:97}\footnote{This is no longer the case in general if $b$ is not continuous, as proved in~\cite{Leb:92-JEDP}.}.}
As a consequence, when (GCC) is not satisfied, we cannot expect a decay of the energy which is uniform with respect to all data in $H^1(M)\times L^2(M)$. However, Lebeau~\cite{Leb:X,Leb:96} proved that there is always a uniform decay rate of the energy, with respect to smoother data, say in $H^2(M)\times H^1(M)$. This motivates the following definition.
\begin{definition}
{
Given $a \in \R$ and a decreasing function $f: [a, +\infty) \to \R_+^*$ such that $f(t)\to 0$ as $t \to +\infty$,
we say that the solutions of~\eqref{eq: stabilization} decay at rate $f(t)$ if there exists $C>0$ such that for all $(u_0, u_1) \in H^2(M)\times H^1(M)$, for all $t\geq a$, we have
$$
E(u(t))^{\frac12} \leq C f(t) \bigl(\|u_0\|_{H^2(M)} + \|u_1\|_{H^1(M)}\bigr) .
$$}
\end{definition}
Note that decay at a rate $f(t)$ depends only on $(M,g)$ and on the damping function $b$. Note also that $f(t)^2$ characterizes the decay of the energy and $f(t)$ that of the associated norm.
Lebeau~\cite{Leb:X,Leb:96} proved that decay at rate $1/\log t$ always holds, independently of $(M,g)$ and $b$ as soon as $\omega_b \neq \emptyset$.
\vs
As noticed for instance in~\cite{BD:08}, decay at a rate $f(t)$ implies faster decay for ``smoother'' data: taking for example $b \in \Cinf(M)$, decay at rate $f(t)$ implies that for all $s>0$, there exists $C_s>0$ such that for all $(u_0, u_1) \in H^{s+1}(M)\times H^s(M)$, we have
$$
E(u(t))^{\frac12} \leq C_s f(t/s)^s (\|u_0\|_{H^{s+1}(M)} + \|u_1\|_{H^s(M)}) .
$$
\vs
{In view of the Rauch-Taylor theorem {mentioned above}, it is convenient to introduce the subset of phase-space consisting in points-directions that are never brought into the damping region $\omega_b$ by the geodesic flow. Namely, the {\em undamped set} is defined by
$${
S = \{\rho \in S^*M, \text{ for all } t \in \R , \ \phi_t(\rho)  \cap T^*\omega_b =\emptyset \} ,}
$$ 
where $\phi_t$ is the geodesic flow (see for instance Appendix~\ref{s:geometry}). With this definition, (GCC) is equivalent to $S=\emptyset$.
In this article, we are concerned with the damped wave equation in a geometric situation where the undamped set $S$ is the cotangent space to a {\em flat subtorus} of $M$ (of dimension $1 \leq n'' \leq n-1$) under two main additional assumptions: the metric is locally flat around this subtorus; the damping function $b$ only depends on variables transverse to this torus and enjoys locally a prescribed homogeneity.
As a particular case, we can consider situations where the geodesic flow has a {\em single undamped trajectory} if the metric is locally flat around this trajectory; the damping function $b$ only depends on variables transverse to the flow and enjoys a prescribed homogeneity in a \nhd of the undamped trajectory.
Such situations may for instance occur on the torus $M=\T^n = (\R/2 \pi \Z)^n$ endowed with the flat metric, as in the following examples.
One of our motivations is to understand the best decay rate in the following model problems.
\begin{example}
\label{ex:motivation}
Let  $M =\T^2 = (\R /2\pi \Z)^2 \equiv [-\pi , \pi]^2$, endowed with the flat metric, let
$\gamma>0$, and let $b(x_1, x_2) = x_1^{2\gamma}$ near $x_{1}=0$, positive elsewhere, depending only on $x_{1}$. 
The undamped set consists in two undamped trajectories: 
$$
S =  \{0\}_{x_1}  \times \T^1_{x_2} \times \{0\}_{\xi_1}  \times \{\pm 1\}_{\xi_2}  = S^*(\{0\} \times \T^1) .
$$
For the case where $b=\sin^{2}x_{1}$, Wen Deng communicated to us a  direct study in \cite{dengperso}.
\end{example}
\begin{example}
\label{ex:motivation2}
Let  $M =\T^3 = (\R /2\pi \Z)^3 \equiv [-\pi , \pi]^3$, endowed with the flat metric, let
$\gamma>0$, and let $b(x_1, x_2, x_3) = (x_1^2 + x_2^2)^{\gamma}$. The undamped set consists in two undamped trajectories: 
$$
S =  \{0, 0\}_{x_1, x_2}  \times \T^1_{x_3} \times \{0, 0\}_{\xi_1, \xi_2}  \times \{\pm 1\}_{\xi_3}  = S^*(\{0,0\} \times \T^1) .
$$
\end{example}
\begin{example}
\label{ex:motivation3}
Let  $M =\T^3 \equiv [-\pi , \pi]^3$, endowed with the flat metric, let
$\gamma>0$, and let $$b(x_1, x_2, x_3) = x_1^{2\gamma}.$$ The undamped set is a $2$-dimensional subtorus given by: 
$$
S =  \{0\}_{x_1}  \times \T^2_{x_2, x_3} \times \{0 \}_{\xi_1}  \times \{(\xi_2,\xi_3) \in \R^2 , \xi_2^2 + \xi_3^2 = 1\}  = S^*(\{0\} \times \T^2) .
$$
\end{example}
}
Decay rates for the damped wave equation on a flat metric with a lack of (GCC) have already been studied in~\cite{LR:05, BH:07, Phung:07b,AL:13}.
In the papers~\cite{LR:05, BH:07,AL:13} it is proved that, on $M =\T^n$, decay at a rate $t^{-1/2}$ always occurs. On the other hand, it is proved in~\cite{AL:13} that the decay cannot be better than $t^{-1}$ as soon as (GCC) is strongly violated, \ie as soon as there exists a {\nhd $\mathcal{N}$ of a geodesic such that $\mathcal{N}\cap \supp(b) = \emptyset$. }
 In this paper, we are studying the opposite situation, \ie the case of a
weak lack of damping on $M =\T^n$: { only a positive codimension invariant torus is undamped.}
In the situation of Examples~\ref{ex:motivation},~\ref{ex:motivation2}, and~\ref{ex:motivation3}, for instance, we may expect (and we shall prove) a decay at a stronger polynomial rate than $t^{-1}$.
Functions on $\T^{n''}$ shall be identified in the whole paper with $2\pi\Z^{n''}$-periodic functions on $\R^{n''}$.
\vs
According to~\cite{Leb:96,BD:08,BT:10,AL:13}, proving a decay rate for solutions of~\eqref{eq: stabilization} reduces to proving a high-energy estimate for the operators 
\begin{equation}\label{plambda}
P_\lambda =  -\Delta_g - \lambda^2 + i \lambda b , \quad \lambda \in \R^*,  \quad D(P_\lambda) = H^2(M).
\end{equation}
The latter are for instance obtained by performing a Fourier transform in the time variable of the damped wave operator $\d_t^2 - \Delta_g + b(x)\d_t$, $\lambda$ being the frequency variable dual to the time $t$.
More precisely, concerning polynomial decay, the optimal result was proved by \cite{BT:10} (see also~\cite{BCT:14} for a simpler proof and generalizations) and can be stated as follows (see \cite[Proposition~2.4]{AL:13}).
\begin{proposition}
 \label{prop:resolvent estimate}
Let $\alpha >0$. Then, the solutions of~\eqref{eq: stabilization} decay at rate
 $
t^{-\frac{1}{\alpha}}
 $
if and only if 
there exist $C,\lambda_{0}$ positive,
such that for all $u\in H^2(M)$, for all $\lambda\ge \lambda_{0}$,
we have 
\begin{equation}
\label{eq:estim a prouver}
C \|P_\lambda u \|_{L^2(M)} \geq \lambda^{1- \alpha} \|u \|_{L^2(M)}.
\end{equation}
\end{proposition}
Recall that  uniform decay is equivalent to the estimate \eqref{eq:estim a prouver} with $\alpha =0$ (and hence to (GCC)). 
\subsection{Main results}
We first have a negative result.
\begin{theorem}
\label{th: sequence saturation}
{Assume that there exists $1\leq n'' \leq n-1$, $\eps_0>0$, and $C_1>0$ such that with  $n' = n - n''$, we have
\begin{align}
\label{eq: local flat metric}
\bullet\ &\left( B_{\R^{n'}}(0, \eps_0)\times \T^{n''} , |dx_1'|^2 + \cdots + |dx_{n'}'|^2 + |dx_1''|^2 + \cdots + |dx_{n''}''|^2 \right) \subset (M, g), \\
\label{eq: local x2 indep}
\bullet\ &\nabla_{x''} b = 0\text{ in }\mathcal{N} = B_{\R^{n'}}(0,\eps_0)\times \T^{n''}, \\
\label{eq: local homo gamma}
\bullet\ & 0 \leq b(x') \leq C_1|x'|^{2\gamma}  
 \text{ in } \mathcal{N} .
\end{align}
Then, there exist $C_0>0$ and $(u_k)_{k \in \N} \in H^2(M)^{\N}$ with $\|u_k \|_{L^2(M)}= 1$ such that
$$
\|P_k u_k\|_{L^2(M)} \leq  C_0 k^{\frac{1}{\gamma+1}}  , \quad \text{for } k \in \N^* .
$$}
\end{theorem}
As a consequence, the best  estimate we could expect is 
\begin{equation}
\label{eq:expected estimate}
C \|P_\lambda u \|_{L^2(M)} \geq \lambda^{\frac{1}{\gamma+1}} \|u \|_{L^2(M)}, 
\end{equation}
\ie \eqref{eq:estim a prouver} with $\alpha =1- \frac{1}{\gamma+1}$. 
Moreover, (see {also}~\cite[Proposition~3]{BD:08}),
our
 Theorem~\ref{th: sequence saturation} prevents decay at a rate $o\big( t^{-(1+\frac{1}{\gamma})}\big)$: the best expected decay rate is $$t^{-(1+\frac{1}{\gamma})}.$$
{
Let us now state our partial converse of Theorem~\ref{th: sequence saturation}: under some additional global assumptions on $M$ and $b$, decay at rate $t^{-(1+\frac{1}{\gamma})}$ indeed holds. We first provide a simpler result in the case $n''=1$ under a global invariance assumption on $b$. We then give our more general result in Theorem~\ref{th: resolvent estimate}.
}
\begin{theorem}
\label{th:b-xn-indep}
{Let  $(M,g)=(\Mp \times \T^1 ,g' + |dx_n|^2 )$ where $(\Mp,g')$ is a smooth compact Riemannian manifold of dimension $n-1$. Assume that there exist 
$y' \in \Mp$, $C_1\geq 1$ and a \nhd $\mathcal{N}'$ of $y'$ such that 
\begin{align}
\label{eq: flat metric}
\bullet\ &g' = |dx_1|^2+ \dots + |dx_{n-1}|^2 \ \ \text{is flat in } \mathcal{N}', \\
\label{eq: x2 indep}
\bullet\ &b = b \otimes 1 \ \ \text{does not depend on the variable } x_n \in\T^1,  \text{and } b \in L^\infty(\Mp), \\
\label{eq: homo gamma}
\bullet\ &C_1^{-1} |x'-y'|^{2\gamma} \leq b(x') \leq C_1 |x'-y'|^{2\gamma} \text{ for } x' \in \mathcal{N}' ,\\
\label{eq: x1 gamma}
\bullet\ & b \geq C_1^{-1}  \text{ a.e. on } \Mp \setminus \mathcal{N}'.
\end{align}
Then, Property \eqref{eq:estim a prouver} holds  with $\alpha =1- \frac{1}{\gamma+1}$, 
\ie decay occurs at rate $t^{-(1+\frac{1}{\gamma})}$.
}
\end{theorem}
This theorem tackles in particular the case of Examples~\ref{ex:motivation} and~\ref{ex:motivation2}. Note that simple examples of functions satisfying the assumptions are given by $b(x') = Q(x'-y')^\gamma$ locally around $y'$, where Q is a definite positive quadratic form. {We stress that we require very  little regularity for the damping coefficient
$b$: its ``vanishing rate'' is prescribed here~\eqref{eq: homo gamma} in a relatively weak sense. One may however discuss its global invariance property in the $x_n$-direction.}
It can indeed be removed: Theorem~\ref{th:b-xn-indep} is a particular case of the following result, where  $\T^1$ is replaced by $\T^{n''}$ (adding no significant difficulty) and $b$ is not supposed to be globally invariant anymore, but instead satisfies (GCC) outside the undamped trajectory. We presented Theorem~\ref{th:b-xn-indep} separately as its proof is simpler and contains
nevertheless the key ideas for the next result.
{
\begin{theorem}
\label{th: resolvent estimate}
{Take $1\leq n'' \leq n-1$ and assume that $(M,g)=(\Mp \times \T^{n''} ,g' + |dx_1''|^2+ \dots + |dx_{n''}''|^2  )$ where $(\Mp,g')$ is a smooth compact Riemannian manifold of dimension $n' = n - n''$ and $(x_1'' , \dots , x_{n''}'')$ denote variables in $\T^{n''}$. Assume that there exist $y' \in \Mp$, $C_1 \geq 1$ and a \nhd $\mathcal{N}'$ of $y'$ such that 
\begin{align}
\label{eq: flat metric bis}
\bullet\ &g' = |dx_1'|^2+ \dots + |dx_{n'}'|^2 \ \ \text{is flat in } \mathcal{N}', \\
\label{eq: regularite b}
\bullet\ &b \in L^\infty (M) \text{ and } \nabla_{x''} b \in L^\infty(M), \\
\label{eq: homo gamma++}
\bullet\ & C_1^{-1} |x'-y'|^{2\gamma} \leq b(x') \leq C_1 |x'-y'|^{2\gamma}  \text{ for } x' \in \mathcal{N}' ,\\
\bullet\ &\text{any geodesic starting from $S^*M \setminus S^*(\{y'\}\times \T^{n''} )$ intersects $\omega_b$ in finite time}   \label{eq: GCC outside}.
\end{align}
Then, we have the property \eqref{eq:estim a prouver} with $\alpha =1- \frac{1}{\gamma+1}$, 
\ie decay at rate $t^{-(1+\frac{1}{\gamma})}$.
}
\end{theorem}
\begin{remark}
The proof of this theorem (as well as those of the previous ones) also holds without significant modification if the square torus $\T^{n''} = (\R/2\pi\Z)^{n''}$ is replaced by the rectangular torus $(\R/L_1\Z) \times \dots \times (\R/L_{n''}\Z)$, or the rectangle $L_1 \times \dots \times L_{n''}$ with Dirichlet or Neumann boundary conditions. {It remains also essentially unchanged if $b$ vanishes near {\em finitely many} points $y'_1, y'_2, \cdots$ (instead of a single one $y'$) assuming Assumptions~\eqref{eq: flat metric bis}-\eqref{eq: GCC outside} around each point (with possibly different vanishing rates $\gamma_1, \gamma_2 , \cdots$, in which case the decay rate is given by $t^{-(1+\frac{1}{\max \gamma_i})}$).}
\end{remark}
}
This result applies for instance on the torus: assume $M =\T^n$ and that there is a single undamped trajectory $\Gamma$. Assume that there exists a \nhd  $\mathcal{N}$ of this trajectory such that $b$ is invariant in $\mathcal{N}$ in the direction of $\Gamma$, and that it is positive homogeneous of degree $2\gamma$ in $\mathcal{N}$ in variables orthogonal to $\Gamma$. Then, we have the property \eqref{eq:estim a prouver} with $\alpha =1- \frac{1}{\gamma+1}$, \ie decay at rate $t^{-(1+\frac{1}{\gamma})}$.
\vs
Since the work of Lebeau~\cite{Leb:96} (see also the introduction of~\cite{AL:13} and the references therein), it is quite well established that the main parameters governing the decay rates when (GCC) fails are the global and local dynamics of the geodesic flow.
Our results confirm the idea, raised in~\cite{BH:07,AL:13}, that once the geometry (and hence the dynamics) is fixed, the next relevant feature when regarding the best decay rate is the rate at which the damping coefficient $b$ vanishes.
Whether the additional global product structure assumption on the manifold $M$ is necessary remains an open question.
\vs
Observe that the bigger $\gamma$,
 the worse is Estimate~\eqref{eq:expected estimate}. This is consistent with the fact that for large $\gamma$, the function $b$ is very flat on $\{y'\}\times \T^{n''}$ so that much energy may keep concentrated  on the set where $b$ is small. 
Note that formally, when taking $\gamma \to 0^+$ in Estimate~\eqref{eq:expected estimate} (and forgetting that the constant $C$ we obtain depends on $\gamma$), we recover the uniform decay estimate (\ie \eqref{eq:estim a prouver} with $\alpha =0$), equivalent to (GCC). Indeed, if $b$ is positive homogeneous of order zero, it does not vanish at $y'$ so that (GCC) is satisfied. It would certainly be interesting to prove Estimate~\eqref{eq:expected estimate} with a constant $C$ uniform with respect to $\gamma$ to make this remark rigorous.
\vs 
The plan of the article is as follows.
Taking advantage of the homogeneity of $b$, the sought estimate near the undamped set may be reduced to an estimate on $\R^{n'}$ for some non-selfadjoint operator. This key estimate is proved in Section~\ref{sec:estimateRd}.  Section~\ref{sec:2lemmata} is devoted to the proof of two simple technical lemmata, one of them being the scaling argument. The proof of Theorem~\ref{th:b-xn-indep} is given in Section~\ref{sec:inariantcase}. The proof of the main result, namely Theorem~\ref{th: resolvent estimate}, is completed in Section~\ref{sec:proofmainth} in two steps: first, we prove a geometric control lemma in Section~\ref{s:geometriccontrol}.
Then, in Section~\ref{sec:patchingestimates}, we patch together the estimates obtained in the different microlocal regions. Section~\ref{sec:quasimodes} provides a proof of the lower bound of Theorem~\ref{th: sequence saturation}. 
 In Section~\ref{sec:2ndmicro}, we discuss the spirit of the proof, which relies on some kind of second microlocalization. 
 In particular, our proof could not work with a standard semi-classical localization procedure: we are left with a region in the phase space, near the undamped set $S$, where further cutting of the phase space is necessary, with a stopping procedure linked to the Heisenberg Uncertainty Principle.
 To patch together the estimates, we use implicitly a metric which should satisfy
 some admissibility properties.
 Although we have avoided in the main part of the text
 to resort
 to very general tools of pseudodifferential calculus, we hope  that Section~\ref{sec:2ndmicro} could  bring a more conceptual vision of the technicalities included in the 
previous sections. The paper ends with three appendices recalling some  facts of geometry and pseudodifferential calculus.
\vs
\noindent
\textbf{Acknowledgements.}
The first author is partially supported by the Agence Nationale de la Recher\-che under grant GERASIC ANR-13-BS01-0007-01.
\section{
A sharp estimate for a non-selfadjoint operator on 
\texorpdfstring{$\R^d$}{rd}}
\label{sec:estimateRd}
\subsection{Statements}
After a Fourier transformation in the periodic direction and a scaling argument (see the following sections), our main result is reduced to the following theorem. We define on $L^2(\R^d)$ (below, we shall take $d=n'$) the unbounded operator
{\begin{equation}\label{qzero}
Q_0^\lambda = - \Delta + i W_\lambda(x), \quad \lambda >0,
\end{equation}
where $W_\lambda$ is a family of real-valued measurable functions and 
$D(Q_{0}^\lambda)=\{u\in H^{2}(\R^{d}), W_\lambda u\in L^{2}(\R^{d})\}.$
\begin{theorem}
\label{th: estimate R^d}
Suppose that $W_\lambda$ is a family of real-valued measurable functions on $\R^d$ and 
that there exist $C_1\ge 1$ and $\gamma>0$ such that for all $\lambda >0$, we have
\begin{equation}
\label{eq: hyp V}
C_1^{-1} |x|^{2\gamma} \leq W_\lambda(x) \leq C_1 \langle x \rangle^{2\gamma}=C_{1}(1+\val x^{2})^{\gamma} .
\end{equation}
Then, there exists $C_0>0$ such that for all $\mu \in \R$, all $\lambda>0$ and $u \in \Con^2_c(\R^d)$, we have
\begin{align} 
\label{eq:estimate Rd}
C_0 \|(Q_0^\lambda - \mu) u \|_{L^2(\R^d)} \geq \left(
 \mu^{\frac{\gamma}{2\gamma + 1}} 
\mathbf{1}{(\mu\geq 1)}
+\val \mu\mathbf{1}{(\mu\le - 1)}+1
\right)\|u \|_{L^2(\R^d)} .
\end{align}
\end{theorem}
We stress the fact that the sole uniform Assumption~\eqref{eq: hyp V} yields the uniform estimate~\eqref{eq:estimate Rd}.}
The power $\frac{\gamma}{2\gamma + 1}$ is optimal in this estimate. Although not needed for the application to the damped wave equation, we provide for completeness a direct proof of this fact in Lemma~\ref{lem:converseestimRd}.
The papers by E. B. Davies \cite{MR1700903}
and K. Pravda-Starov \cite{MR2241978}
gave a version of the above estimates in the case of the 1D complex harmonic oscillator,
$-\frac{d^{2}}{dx^{2}}+e^{i\theta}x^{2}$.
{\vs
To prove Theorem~\ref{th: estimate R^d}, we need the following lemma.
\begin{lemma}
\label{lem: estimee restes}
Suppose that $W_\lambda$ satisfy the uniform Assumption \eqref{eq: hyp V} and let  $a$ be a smooth function on $\R^{2d}$, bounded as well as all its derivatives. Then, there exists $C>0$ such that for all $\lambda >0$ and all $u \in \Con^0_c(\R^d)$, we have
$$
\|V_\lambda a^w u\|_{L^2(\R^d)} \leq C\left(  \|u \|_{L^2(\R^d)}  + \|V_\lambda u \|_{L^2(\R^d)}  \right),
$$
where $V_\lambda=W_\lambda^{1/2}$ and $a^{w}$ stands for the Weyl quantization of the symbol $a$.
\end{lemma}
\begin{proof}[Proof of Lemma~\ref{lem: estimee restes}]
Using the upper bound in Assumption \eqref{eq: hyp V} yields 
\begin{align*}
\|V_\lambda (x) a^w u\|_{L^2(\R^d)}^2 & = \int_{\R^d} V_\lambda^2(x) |a^w u|^2 dx \\
&  \leq C_1 \int_{\R^d} \langle x \rangle^{2\gamma} |a^w u|^2 dx = C_1 \|\langle x \rangle^{\gamma} a^w u\|_{L^2(\R^d)}^2 .
\end{align*}
Then, we notice that $\langle x \rangle^{\gamma}$ and $\langle x \rangle^{-\gamma}$ are admissible weight functions for the metric $|dx|^2 + |d\xi|^2$ in the sense of~\cite[Definition~2.2.15]{Lernerbook}.
 As a consequence of symbolic calculus, we have 
 $$\langle x \rangle^{\gamma} \sharp a \sharp \langle x \rangle^{-\gamma} \in S(1, |dx|^{2}+ |d\xi|^{2}),
 $$ 
where 
 $
 S(1, |dx|^{2}+ |d\xi|^{2})
 $
 is the space of smooth functions on $\R^{2d}$ which are bounded as well as all their derivatives
 (see Section \ref{sec.pseudo} in the Appendix for more on this topic).
Calder\'on-Vaillancourt theorem (see e.g. ~\cite[Theorem~1.1.4]{Lernerbook}) yields 
 $$\langle x \rangle^{\gamma} a(x, \xi)^w \langle x \rangle^{-\gamma} \in \calL (L^2(\R^d)),$$ which implies
$
\|\langle x \rangle^{\gamma} a^w u\|_{L^2(\R^d)} \lesssim \|\langle x \rangle^{\gamma} u\|_{L^2(\R^d)} .
$
This finally gives 
\begin{align*}
\|V_\lambda (x) a(x, \xi)^w u\|_{L^2(\R^d)}^2 & \lesssim \|\langle x \rangle^{\gamma} u\|_{L^2(\R^d)}^2
\lesssim \| u\|_{L^2(\R^d)}^2 + \| |x|^{\gamma} u\|_{L^2(\R^d)}^2 \\
&\lesssim \| u\|_{L^2(\R^d)}^2 + \| V_\lambda  u\|_{L^2(\R^d)}^2,
\end{align*}
according to the uniform lower bound in Assumption \eqref{eq: hyp V}. This concludes the proof of the lemma.
\end{proof}

Now, the proof of Theorem~\ref{th: estimate R^d} follows from the next two lemmata.
\begin{lemma}
\label{th: estimate R^d mu large}
There exists $C>0$ and $\mu_0 \geq 0$ such that for all $\mu \geq \mu_0$, all $\lambda>0$ and $u \in \Con^2_c(\R^d)$, we have
\begin{align} 
\label{eq:estimate Rd mu large}
C \|(Q_0^\lambda - \mu) u \|_{L^2(\R^d)} \geq 
 \mu^{\frac{\gamma}{2\gamma + 1}} \|u \|_{L^2(\R^d)} .
\end{align}
\end{lemma}

\begin{lemma}
\label{th: estimate R^d mu small}
For any $\mu_0 \geq 0$, there exists $C>0$ such that for all $\mu \leq \mu_0$, all $\lambda>0$ and $u \in \Con^2_c(\R^d)$, we have
\begin{align} 
\label{eq:estimate Rd mu small}
C \|(Q_0^\lambda - \mu) u \|_{L^2(\R^d)} \geq 
(\val \mu\mathbf{1}{(\mu\le - 1)}+1)
\|u \|_{L^2(\R^d)} .
\end{align}
\end{lemma}
Let us first prove the simpler Lemma~\ref{th: estimate R^d mu small}, the proof of the more involved  Lemma~\ref{th: estimate R^d mu large} being postponed to the end of the section.
\begin{proof}[Proof of Lemma~\ref{th: estimate R^d mu small}]
{We start with the case
$\mu\le -1$.}
We  have then
$$
\|(Q_0^\lambda-\mu) u\|_{L^2(\R^d)} \|u\|_{L^2(\R^d)} \geq \Re \poscal{(Q_0^\lambda-\mu) u}{u}_{L^2(\R^d)} 
\geq - \mu \poscal{u}{u}_{L^2(\R^d)} 
=\val \mu \| u\|_{L^2(\R^d)}^2 ,
$$
so that Estimate~\eqref{eq:estimate Rd} holds for $\mu \leq -1$.
\vs
Next, let us prove that there exists $C>0$ such that for all $\mu \in [-1 , \mu_0]$, all $\lambda>0$ and $u \in \Con^2_c(\R^d)$, we have
\begin{align} 
\label{eq:estimate Rd mu interm}
C \|(Q_0^\lambda - \mu) u \|_{L^2(\R^d)} \geq 
\|u \|_{L^2(\R^d)} .
\end{align}
If not, there exist sequences $(\lambda_k)_{k\in \N} \in (\R_+)^\N$, $(\lambda_k)_{k\in \N} \in [-1 , \mu_0]^\N$ and $(u_k)_{k\in \N} \in \Con^2_c(\R^d)^\N$ such that 
\begin{equation}
\label{eqn: contradiction u=1}
\|(Q_0^{\lambda_k} - \mu_k) u_k \|_{L^2(\R^d)} < \frac{1}{k+1} , \quad 
\|u_k \|_{L^2(\R^d)} = 1 .
\end{equation}
This implies
\begin{align}
0\leftarrow \|( Q_0^{\lambda_k} -\mu_k )u_k\|_{L^2(\R^d)} \|u_k\|_{L^2(\R^d)} & \geq \Re \poscal{(Q_0^{\lambda_k} -\mu_k)u_k}{u_k}_{L^2(\R^d)} = \|\nabla u_k\|_{L^2(\R^d)}^2  -  \mu_k \| u_k\|_{L^2(\R^d)}^2 \nonumber  \\
& \geq \|\nabla u_k\|_{L^2(\R^d)}^2  -  \mu_0 \| u_k\|_{L^2(\R^d)}^2
\nonumber \\
 0\leftarrow \|( Q_0^{\lambda_k} -\mu_k )u_k\|_{L^2(\R^d)} \|u_k\|_{L^2(\R^d)} & \geq  \Im \poscal{Q_0 u}{u}_{L^2(\R^d)} 
=  \|V_\lambda u_k\|_{L^2(\R^d)}^2 \nonumber \\
&\geq C_1^{-1}  \||x|^{\gamma} u_k\|_{L^2(\R^d)}^2  .
\label{Vu to zero}
\end{align}
Since $H^1_\gamma (\R^d) := \{u \in H^1(\R^{d}), |x|^{\gamma} u \in L^2(\R^{d})\}$ injects compactly in $L^2(\R^d)$ and $(u_k)_{k\in \N}$ is bounded in $H^1_\gamma (\R^d)$, we may extract a subsequence such that $(u_k)_{k\in \N}$ converges strongly in $L^2(\R^d)$, $u_k \to u_\infty \in L^2(\R^d)$. According to~\eqref{Vu to zero}, we also obtain $u_k \rightharpoonup 0$ weakly in $L^2(\R^d)$:
{in fact we have for $\phi\in C^{0}_{c}(\R^{d}\backslash\{0\})$,
$$
\Val{\int u_{\io}(x)\phi(x) dx}\leftarrow\Val{\int u_{k}(x)\phi(x) dx}\le \norm{\val{x}^{\gamma}u_{k}}_{L^{2}(\R^{d})}\norm{\val{x}^{-\gamma}\phi}_{L^{2}(\R^{d})}\rightarrow 0,
$$
proving that $\supp u_{\io}\subset\{0\}$ and thus the $L^{2}$ function $u_{\io}=0$.
}
This contradicts~\eqref{eqn: contradiction u=1} which implies $\|u_\infty\|_{L^2(\R^d)}=1$. This proves~\eqref{eq:estimate Rd mu interm}. As a consequence, Estimate~\eqref{eq:estimate Rd} is now proven to hold for all $\mu \in (-\io, \mu_{0}]$.
\end{proof}
We are now left to prove Lemma~\ref{th: estimate R^d mu large},
\ie to study the most substantial case where $\mu>\mu_{0}$, but we may keep in mind 
that we can freely choose the large fixed constant $\mu_{0}$.
We  set 
\begin{equation}\label{notnu}
\nu=\mu^{1/2},\quad Q_\nu^\lambda  = Q_0^\lambda - \nu^2,
\end{equation}
 and study the asymptotics when $\nu \to +\infty$. From the above remarks, we have only to prove the estimate
 \eqref{eq:estimate Rd} for $\nu\ge \nu_{0}$, where $\nu_{0}$ can be chosen arbitrarily large.
 First of all, we note  that
\begin{equation}
\label{eq: ellipticdamping}
\|Q_\nu^\lambda  u\|_{L^2(\R^d)} \|u\|_{L^2(\R^d)} \geq \Im\poscal{Q_\nu^\lambda  u}{u}_{L^2(\R^d)}= \poscal{W_\lambda u}{u}_{L^2(\R^d)}\ge C_{1}^{-1} \poscal{\val x^{2\gamma}u}{u}_{L^2(\R^d)}, 
\end{equation}
which will be used several times during the proof.
In particular, this estimate provides  the right scale in the region $|x|\geq \nu^{{1}/{(2\gamma +1)}}$, according to the lower bound in Assumption \eqref{eq: hyp V}. Next, we split the phase space in two different regions.
\subsection{\bfseries The propagative region} 
 Let $\chi\in\Cinfc(\R^+;[0,1])$, such $\chi=1$ on $[1/2,3/2]$ and $\chi=0$ on $[0,1/4]$. Let $\varphi \in \Cinfc(\R^d ;[0,1])$ such that $\varphi (x) = \frac12$ if $|x| \leq \frac12$ and $\varphi (x) = 0$ if $|x| \geq 1$. We define 
 $$\psi(x, \xi) = \int_{-\infty}^0\varphi(x + \tau \xi) d\tau \in \Cinf\bigl(\R^d \times (\R^d\setminus \{0\})\bigr),$$ which is bounded on $\R^d \times \bigl(\R^d\setminus B(0,\frac14)\bigr)$ since $\varphi$ is compactly supported.
We set 
$$m_\nu (x, \xi)= \chi\bigl(\frac{|\xi|^{2}}{\nu^{2}}\bigr) \psi(\frac{x}{\nu^{\frac{1}{2\gamma+1}}}, \frac{\xi}{\nu})\in S(1 , \frac{|dx|^{2}}{\nu^{\frac{2}{2\gamma+1}}} + \frac{|d\xi|^{2}}{\nu^{2}}),$$
where each  seminorm of the symbol $m_{\nu}$ is  bounded above independently of $\nu\ge 1$;
in particular, we get 
that  $m_\nu^w$ is bounded on $L^2(\R^d)$ with $\sup_{\nu\ge 1}\norm{m_{\nu}^{w}}_{\mathcal L(L^{2})}<+\io$. Next, we have
\begin{align}
\label{eq: multiplier 1}
2 \Re \poscal{Q_\nu^\lambda  u }{i m_\nu^w u}_{L^2(\R^d)} &= 
\Poscal{ i \left[ (|\xi|^2-\nu^2)^w , m_\nu^w \right] u }{ u }_{L^2(\R^d)} +
 2 \Re \poscal{V_\lambda^2  u} {m_\nu^w u}_{L^2(\R^d)} \nonumber \\
&= \poscal
{\left\{ |\xi|^2-\nu^2 , m_\nu \right\}^w u}{u}_{L^2(\R^d)} + 2 \Re \poscal{V_\lambda^2  u}{m_\nu^w u}_{L^2(\R^d)} ,
\end{align}
since the symbol $|\xi|^2-\nu^2$ is quadratic. Moreover, we can compute
$$
\left\{ |\xi|^2-\nu^2 , m_\nu \right\} = 2 \xi \cdot \d_x m_\nu = 2 \chi(\frac{|\xi|^{2}}{\nu^{2}}) \xi \cdot \frac{\p}{\p x}\left( \psi(\frac{x}{\nu^{\frac{1}{2\gamma+1}}}, \frac{\xi}{\nu}) \right)
$$
with 
\begin{align*}
\chi(\frac{|\xi|^{2}}{\nu^{2}})  \xi \cdot \frac{\p}{\p x}\Bigl( \psi (x \nu^{-\frac{1}{2\gamma+1}}, \xi\nu^{-1})\Bigr)  
&= \int_{-\infty}^0\nu^{-\frac{1}{2\gamma+1}} \chi(\frac{|\xi|^{2}}{\nu^{2}}) (\xi \cdot d\varphi)(x \nu^{-\frac{1}{2\gamma+1}} +\tau \xi\nu^{-1}) d\tau  \\
&= \int_{-\infty}^0 \nu \nu^{-\frac{1}{2\gamma+1}} \chi(\frac{|\xi|^{2}}{\nu^{2}}) \frac{d}{d\tau} \Bigl(\varphi(x \nu^{-\frac{1}{2\gamma+1}} +\tau \xi\nu^{-1}) \Bigr)d\tau \\
&=\nu^{\frac{2\gamma}{2\gamma+1}}\chi(\frac{|\xi|^{2}}{\nu^{2}}) \varphi(x \nu^{-\frac{1}{2\gamma+1}})
\end{align*}
since $| \xi\nu^{-1}|^{2}\geq \frac14$ on $\supp \chi(\frac{|\xi|^{2}}{\nu^{2}})$. 
Hence, we obtain
\begin{multline*}
\left\{ |\xi|^2-\nu^2 , m_\nu \right\} = 2 \nu^{\frac{2\gamma}{2\gamma+1}} \chi(\frac{|\xi|^{2}}{\nu^{2}}) \varphi(x \nu^{-\frac{1}{2\gamma+1}})
\\\ge 
\begin{cases}\nu^{\frac{2\gamma}{2\gamma+1}} \quad
\text{if $\val{|\xi|^{2}\nu^{-2}-1}\le1/2$ and $ |x| \nu^{-\frac{1}{2\gamma+1}}\le 1/2$},\\
0 \text{ on } T^* \R^d.
\end{cases}
\end{multline*}
Moreover we have,
$$
2 \nu^{\frac{2\gamma}{2\gamma+1}} \chi(\frac{|\xi|^{2}}{\nu^{2}}) \varphi(x \nu^{-\frac{1}{2\gamma+1}})\in 
S(\nu^{\frac{2\gamma}{2\gamma+1}} ,\frac{|dx|^{2}}{\nu^{\frac{2}{2\gamma+1}}}+\frac{|d\xi|^{2}}{\nu^{2}}).
$$
As a consequence, using the sharp G{\aa}rding inequality in~\eqref{eq: multiplier 1} yields 
\begin{multline}
\label{eq: garding 1}
C\norm{Q_\nu^\lambda u}_{L^2(\R^d)}\norm{u}_{L^2(\R^d)}\ge
\nu^{\frac{2\gamma}{2\gamma+1}}
\Poscal{\Bigl(\chi_{0}(|\xi|^{2}\nu^{-2}-1)\chi_{0}(|x|^{2}\nu^{-\frac{2}{2\gamma+1}})\Bigr)^{w}u }{u}_{L^2(\R^d)}\\
-|2 \Re \poscal{ V_\lambda^2  u}{m_\nu^w u }_{L^2(\R^d)}|  
- C \nu^{-\frac{2}{2\gamma+1}}\norm{u}_{L^2(\R^d)}^2,
\end{multline}
where, for some $\epsilon_{0}\in (0,1/8)$,
\begin{align}
\label{e:defchi0}
\begin{cases}
\chi_0 \in \Cinfc(\R;[0,1])  \text{ is such that }\{\chi_{0}=1\}=[-\epsilon_{0},\epsilon_{0}], \text{ and }\\
\ \{\chi_{0}=0\}=[-2\epsilon_{0},2\epsilon_{0}]^{c}, \ \{0<\chi_{0}(t)<1\}=\{\epsilon_{0}<\val t<
2\epsilon_{0}\}.
\end{cases}
\end{align}
Next, we check  the term $2 \Re \poscal{V_\lambda^2  u}{m_\nu^w u}_{L^2(\R^d)}$. We have
$$
\left|2 \Re \poscal{V_\lambda^2  u}{m_\nu^w u}_{L^2(\R^d)} \right|
= \left|2 \Re \poscal{V_\lambda u}{V_\lambda m_\nu^w u}_{L^2(\R^d)}\right|
\leq 2 \|V_\lambda u \|_{L^2(\R^d)}\|V_\lambda m_\nu^w u \|_{L^2(\R^d)} .
$$
Recalling that $m_\nu \in S(1 , \frac{|dx|^{2}}{\nu^{\frac{2}{2\gamma+1}}}+ \frac{|d\xi|^{2}}{\nu^{2}}) \subset S(1 ,|dx|^{2} +|d\xi|^{2})$ as $\nu\ge 1$, we may apply Lemma~\ref{lem: estimee restes} to obtain
$$
\Val{\Re \poscal{V_\lambda^2  u}{m_\nu^w u}_{L^2(\R^d)} }
\lesssim \|V_\lambda u \|_{L^2(\R^d)} \left( \|u \|_{L^2(\R^d)}  + \|V_\lambda u \|_{L^2(\R^d)}  \right),
$$
where the constant involved is uniform \wrt $\nu$ and $\lambda$. Next, using~\eqref{eq: ellipticdamping}, we obtain 
\begin{align}
\label{eq:remainder estimate}
\left|2 \Re \poscal{V_\lambda^2  u}{m_\nu^w u}_{L^2(\R^d)} \right| 
& \lesssim  \|u \|_{L^2(\R^d)}^{\frac32}\|Q_\nu^\lambda  u\|_{L^2(\R^d)}^{\frac12} + \|Q_\nu^\lambda  u\|_{L^2(\R^d)} \|u\|_{L^2(\R^d)}  \nonumber \\
& \lesssim  \|u \|_{L^2(\R^d)}^2+ \|Q_\nu^\lambda  u\|_{L^2(\R^d)} \|u\|_{L^2(\R^d)} .
\end{align}
Combining this estimate with~\eqref{eq: garding 1}, we have, for $\nu \geq \nu_0$ and $\nu_0$  large enough,
\begin{equation}
\label{eq: propag region}
C\|u \|_{L^2(\R^d)}^2 + C\norm{Q_\nu^\lambda u}_{L^2(\R^d)}\norm{u}_{L^2(\R^d)}
\ge
\nu^{\frac{2\gamma}{2\gamma+1}}\Poscal{\Bigl(\chi_{0}(|\xi|^{2}\nu^{-2}-1)\chi_{0}(|x|^{2}\nu^{-\frac{2}{2\gamma+1}})\Bigr)^{w}u}{u}_{L^2(\R^d)}.
\end{equation}
\subsection{The elliptic region.}
We now check the regions where $|\xi|^2 \ll \nu^{2}$ or $|\xi|^2 \gg \nu^{2}$. Let $\epsilon_{0}\in (0,1/2)$;
we consider a function $\theta\in \Cinf(\R ;[-1,1])$ such that 
\begin{equation}\label{theta}
\theta(\sigma)=
\begin{cases}
1&\text{for $\sigma\ge 1+2\epsilon_{0}$,}\\
 \{0<\theta <1\}&\text{for } \sigma\in (1+\epsilon_{0},1+2\epsilon_{0}),\\
0&\text{for $1-\epsilon_{0}\le \sigma\le 1+\epsilon_{0}$,}\\
 \{-1<\theta <0\}&\text{for } \sigma\in (1-2\epsilon_{0},1-\epsilon_{0}),\\
-1&\text{for $\sigma\le1-2\epsilon_{0}.$}
\end{cases}
\end{equation}
We claim that 
\begin{equation}\label{triv}
\forall \sigma\in \R,\quad
(\sigma-1)\theta(\sigma)\ge \val{\theta(\sigma)}(\sigma+1)\frac{\epsilon_{0}}{2+\epsilon_{0}}.
\end{equation}
In fact, \eqref{triv} is obvious whenever $\theta(\sigma)=0$ and if $\theta(\sigma)>0$, \ie if $\sigma>1+\epsilon_{0}$, since it amounts to 
verify
$$
\sigma(1-\frac{\epsilon_{0}}{2+\epsilon_{0}})\ge 1+\frac{\epsilon_{0}}{2+\epsilon_{0}}\quad \text{\ie}\quad \sigma\ge 1+\epsilon_{0}, \quad\text{which holds true.}
$$
If $\theta(\sigma)<0$, \ie if $\sigma<1-\epsilon_{0}$,
it amounts to 
verify
$$
\sigma(1+\frac{\epsilon_{0}}{2+\epsilon_{0}})\le 1-\frac{\epsilon_{0}}{2+\epsilon_{0}}\quad \text{\ie}\quad \sigma\le \frac{1}{1+\epsilon_{0}}, \quad\text{which holds true since $1-\epsilon_{0}\le \frac{1}{1+\epsilon_{0}}$.}
$$
A consequence of \eqref{triv} is that, with $c_{0}=\frac{\epsilon_{0}}{2+\epsilon_{0}}$, we have 
 \begin{equation}\label{triv+} \forall \xi\in \R^{d}, \forall \nu\ge 1,\quad
(\val \xi^{2}-\nu^{2})\theta(\val \xi^{2}\nu^{-2})\ge c_{0}\val{\theta(\val \xi^{2}\nu^{-2})}\bigl(\val \xi^{2}+\nu^{2}\bigr).
\end{equation}
We compute 
\begin{multline*}
\re
\poscal{Q_\nu^\lambda u}{\theta(|\xi|^{2}\nu^{-2})^{w}u}_{L^2(\R^d)}
 \\=\Poscal{(|\xi|^{2}-\nu^{2})^{w}u}{\theta(|\xi|^{2}\nu^{-2})^{w}u}_{L^2(\R^d)}
+ \re\Poscal{i V_\lambda ^2 u}{\theta(|\xi|^{2}\nu^{-2})^{w}u}_{L^2(\R^d)} \\
 \ge c_{0} \Poscal{\bigl((|\xi|^{2}+\nu^{2})\val{\theta(|\xi|^{2}\nu^{-2})}\bigr)^{w}u}{u}_{L^2(\R^d)}
- \left|\re\Poscal{ i V_\lambda ^2 u}{\theta(|\xi|^{2}{}\nu^{-2})^{w}u}_{L^2(\R^d)}\right|.
\end{multline*}
 Following~\eqref{eq:remainder estimate}, we have 
\begin{align*}
\left| \re\poscal{ i V_\lambda ^2 u}{\theta(|\xi|^{2}\nu^{-2})^{w}u}_{L^2(\R^d)}\right|
 \lesssim  \|u \|_{L^2(\R^d)}^2+ \|Q_\nu^\lambda  u\|_{L^2(\R^d)} \|u\|_{L^2(\R^d)} , 
\end{align*}
so that we finally obtain, as $\theta(|\xi|^{2}\nu^{-2})^{w}$ is bounded on $L^2(\R^d)$,
\begin{equation}
\label{eq: elliptic region}
C\|u \|_{L^2(\R^d)}^2 + C\norm{Q_\nu^\lambda u}_{L^2(\R^d)}\norm{u}_{L^2(\R^d)}
 \ge \Poscal{ \bigl((|\xi|^{2}+\nu^{2})\val{\theta(|\xi|^{2}\nu^{-2})}\bigr)^{w}u}{u}_{L^2(\R^d)}.
\end{equation}
\subsection{Patching the estimates together.}
Combining \eqref{eq: ellipticdamping}, \eqref{eq: propag region} and \eqref{eq: elliptic region}, we obtain the following estimate
\begin{multline}
\label{eq: 3 regions}
C\|u \|_{L^2(\R^d)}^2 + C\norm{Q_\nu^\lambda u}_{L^2(\R^d)}\norm{u}_{L^2(\R^d)}
 \ge
 \|V_\lambda u\|_{L^2(\R^d)}^2  
 + \nu^{2} \Poscal{\bigl(\val{\theta(|\xi|^{2}\nu^{-2})}\bigr)^{w}u} {u}_{L^2(\R^d)} \\ 
 + \nu^{\frac{2\gamma}{2\gamma+1}}\Poscal
{\Bigl(\chi_{0}(|\xi|^{2}\nu^{-2}-1)\chi_{0}(|x|^{2}\nu^{-\frac{2}{2\gamma+1}})\Bigr)^{w}u}{u}_{L^2(\R^d)} .
\end{multline}
Since  $\chi_{0}$ is given and satisfies
\eqref{e:defchi0},
we define now 
$$
\theta(\sigma)=\sign(\sigma-1)\bigl(1-\chi_{0}(\sigma-1)\bigr).
$$
Since $\chi_{0}$ is smooth and vanishes near 0, the function $\theta$ is smooth
and such that
$$
\begin{cases}
\text{for $\sigma\ge 1+2\epsilon_{0}$,}&\theta(\sigma)=1,\\
\text{for $1+\epsilon_{0}<\sigma< 1+2\epsilon_{0}$,}&\theta(\sigma)\in (0,1),\\
\text{for $1-\epsilon_{0}\le \sigma\le 1+\epsilon_{0}$,}&\theta(\sigma)=0,\\
\text{for $1-2\epsilon_{0}< \sigma< 1-\epsilon_{0}$,}&\theta(\sigma)\in (-1,0),\\
\text{for $\sigma\le 1-2\epsilon_{0}$,}&\theta(\sigma)=-1.
\end{cases}
$$
That function $\theta$ satisfies \eqref{theta} 
so that \eqref{triv} holds.
We note that
 $$\val{\theta(\sigma)} +\chi_{0}(\sigma-1) = 1,$$ 
 since 
 $
\val{1-\chi_{0}(\sigma-1)}+\chi_{0}(\sigma-1)=1-\chi_{0}(\sigma-1)+\chi_{0}(\sigma-1)=1.
 $
 As a consequence, we write
\begin{align*}
1 & = \val{\theta(|\xi|^{2}\nu^{-2})} +\chi_{0}(|\xi|^{2}\nu^{-2}-1)\\
  & = \val{\theta(|\xi|^{2}\nu^{-2})} + \chi_{0}(|x|^{2}\nu^{-\frac{2}{2\gamma+1}})\chi_{0}(|\xi|^{2}\nu^{-2}-1)
 +\bigl(1- \chi_{0}(|x|^{2}\nu^{-\frac{2}{2\gamma+1}})\bigr)\chi_{0}(|\xi|^{2}\nu^{-2}-1),
\end{align*}
and hence
\begin{align*}
\nu^{\frac{2\gamma}{2\gamma+1}} & \le \nu^{\frac{2\gamma}{2\gamma+1}}\val{\theta(|\xi|^{2}\nu^{-2})}
 + \nu^{\frac{2\gamma}{2\gamma+1}}\chi_{0}(|x|^{2}\nu^{-\frac{2}{2\gamma+1}})\chi_{0}(|\xi|^{2}\nu^{-2}-1)
 + \nu^{\frac{2\gamma}{2\gamma+1}}(1- \chi_{0}(|x|^{2}\nu^{-\frac{2}{2\gamma+1}})) .
\end{align*}
Since the symbols on both sides of the inequality belong to the class $$S\bigl(\nu^{\frac{2\gamma}{2\gamma+1}} ,\frac{|dx|^{2}}{\nu^{\frac{2}{2\gamma+1}}}+\frac{|d\xi|^{2}}{\nu^{2}}\bigr),$$
 we can apply G{\aa}rding's inequality. 
Note that the gain in the pseudodifferential calculus for symbols in this class is given by $\nu^{-\frac{1}{2\gamma+1}} \nu^{-1} = \nu^{-\frac{2\gamma+2}{2\gamma+1}}$.
This gives, for $\nu \geq \nu_0$ and $\nu_0$  large enough,
\begin{multline*}
  \nu^{\frac{2\gamma}{2\gamma+1}}
  \Poscal{\bigl(\chi_{0}(|\xi|^{2}\nu^{-2}-1)\chi_{0}(|x|^{2}\nu^{-\frac{2}{2\gamma+1}})\bigr)^{w}u}{u}_{L^2(\R^d)}
+ \nu^{\frac{2\gamma}{2\gamma+1}} \poscal{\bigl(\val{\theta(|\xi|^{2}\nu^{-2})}\bigr)^{w}u}{u}_{L^2(\R^d)}\\
+ \Poscal{\nu^{\frac{2\gamma}{2\gamma+1}}(1- \chi_{0}(|x|^{2}\nu^{-\frac{2}{2\gamma+1}})) u }{u}_{L^2(\R^d)}
 \ge \frac12 \nu^{\frac{2\gamma}{2\gamma+1}} \|u\|_{L^2(\R^d)}^2  .
\end{multline*}
Next, we note  that, because of the properties of $\chi_0$, given in \eqref{e:defchi0} we find
$$
\nu^{\frac{2\gamma}{2\gamma+1}}(1- \chi_{0}(|x|^{2}\nu^{-\frac{2}{2\gamma+1}})) \leq C V_\lambda (x)^2 ,
$$
according to the lower bound in Assumption~\eqref{eq: hyp V}.
Using the last two inequalities together with~\eqref{eq: 3 regions} gives
\begin{equation}
\label{eq: 3 regions final}
C\|u \|_{L^2(\R^d)}^2 + C\norm{Q_\nu^\lambda u}_{L^2(\R^d)}\norm{u}_{L^2(\R^d)}
 \ge 
\nu^{\frac{2\gamma}{2\gamma+1}}\|u\|_{L^2(\R^d)}^2  ,
\end{equation}
which concludes the proof of the theorem, dividing by $\norm{u}_{L^2(\R^d)}$ and taking $\nu \geq \nu_0$ with $\nu_0$ large enough.
}
\section{
Two lemmata}
\label{sec:2lemmata}
In this section, we state and prove two simple technical lemmata that will be used in the proofs of both Theorems~\ref{th:b-xn-indep} and~\ref{th: resolvent estimate}. 
\subsection{Scaling argument}
{First, we prove the following lemma, which is a consequence of Theorem~\ref{th: estimate R^d} together with a scaling argument. 
We define on $L^2(\R^d)$ (below, we shall take $d=n'$) the operator
\begin{equation}\label{ptildelom}
\wt{P}_{\lambda,\omega} = - \Delta - \omega + i \lambda W (x).
\end{equation}
\begin{lemma}
\label{lem: P lambda omega}
Let $\gamma>0$ be given. Assume that there exist $C_1\ge 1$ and $\gamma>0$ such that 
\begin{equation}
\label{eq: hyp WW}
C_1^{-1} |x|^{2\gamma} \leq W(x) \leq C_1 | x |^{2\gamma}  , \quad x \in \R^d . 
\end{equation}
Then, there exists $C>0$ such that for all $u \in \Con^2_c(\R^d)$, for all $\lambda > 0$ and for all $\omega \in \R$, we have 
\begin{align*} 
C \|\wt{P}_{\lambda,\omega} u \|_{L^2(\R^d)} \geq 
\lambda^{\frac{1}{\gamma +1}} \biggl( 1+ H(\omega) \Bigl(\frac{\omega}{\lambda^{\frac{1}{\gamma +1}}}\Bigr)^{\frac{\gamma}{2\gamma + 1}} \biggr)\|u \|_{L^2(\R^d)},
\end{align*}
where  $H=\mathbf 1_{\R_{+}}$ is the Heaviside function.
\end{lemma}
\begin{remark}
Note that this lemma does not use either $\lambda$ large, or $0\leq \omega \leq \lambda^2$. 
\end{remark}
\begin{proof}[Proof of Lemma~\ref{lem: P lambda omega}]
First, we remark that for all $\alpha>0$, the operator
\begin{equation*}
\begin{array}{rcl}
T_\alpha : L^2(\R^d)& \to & L^2(\R^d) \\
u(x)& \mapsto &\alpha^{\frac{d}{2}}u(\alpha x)
\end{array}
\end{equation*}
is an isometry, with inverse $(T_\alpha)^{-1} = T_{\alpha^{-1}}$.
As a consequence, we have
$$
  T_\alpha \tilde{P}_{\lambda,\omega}(T_\alpha)^{-1} 
  = - \alpha^{-2} \Delta - \omega + i \lambda W (\alpha x),
$$
where, according to Assumption~\eqref{eq: hyp WW}, we have
$$
C_1^{-1}\lambda \alpha^{2\gamma} |x|^{2\gamma} \leq \lambda W (\alpha x) \leq C_1 \lambda \alpha^{2\gamma} | x |^{2\gamma}  , \quad x \in \R^d . 
$$
Now, we choose $\alpha = \lambda^{-\frac{1}{2(\gamma+1)}}$, so that we have $\alpha^{-2} = \alpha^{2\gamma} \lambda = \lambda^{\frac{1}{\gamma+1}}$. Setting
$$
W_\lambda (x) := \lambda^{- \frac{1}{\gamma+1}} \lambda W (\alpha x) , 
$$
we obtain the uniform estimates
\begin{equation}
\label{eq: hyp WWbis}
C_1^{-1}  |x|^{2\gamma} \leq  W_\lambda (  x) \leq C_1  | x |^{2\gamma}  , \quad x \in \R^d , \lambda >0 . 
\end{equation}
With $Q_{0}^\lambda$ defined in \eqref{qzero}, this now yields
$$
  T_\alpha \tilde{P}_{\lambda,\omega}(T_\alpha)^{-1} 
  = \lambda^{\frac{1}{\gamma+1}}\left( Q_0^\lambda - \omega\lambda^{-\frac{1}{\gamma+1}}\right) , \qquad \alpha = \lambda^{-\frac{1}{2(\gamma+1)}} .
$$
Since~\eqref{eq: hyp WWbis}
 implies Assumption~\eqref{eq: hyp V},
 we can apply Theorem~\ref{th: estimate R^d} (where $\mu = \omega \lambda^{\frac{1}{\gamma + 1}}$). This yields (still with $\alpha = \lambda^{-\frac{1}{2(\gamma+1)}}$),
\begin{align*}
\| \wt{P}_{\lambda,\omega} u \|_{L^2(\R^d)} &= \| T_\alpha \wt{P}_{\lambda,\omega}(T_\alpha)^{-1}  T_\alpha u\|_{L^2(\R^d)}
= \lambda^{\frac{1}{\gamma+1}}\left\|\left( Q_0^\lambda - \omega\lambda^{-\frac{1}{\gamma+1}}\right) T_\alpha u \right\|_{L^2(\R^d)} \\
&\ge C_0^{-1} \lambda^{\frac{1}{\gamma+1}}
\left( 1  + H(\omega)\left(\frac{\omega}{\lambda^{\frac{1}{\gamma +1}}}\right)^{\frac{\gamma}{2\gamma + 1}} \right) \| T_\alpha u \|_{L^2(\R^d)},
\end{align*}
concluding the proof of the lemma since $T_\alpha$ is an isometry.
\end{proof}
}
\subsection{An elementary
lemma}
Let $\gamma>0, c_{0}>0$ be given.
We define for $\lambda>0,\omega\ge0$,
\begin{equation}
\label{eq: f(lambda,omega)}
f(\lambda, \omega) =1+ \Big(\frac{\omega}{\lambda^{\frac{1}{\gamma +1}}}\Big)^{\frac{2\gamma}{2\gamma + 1}} - c_0 \ \omega \lambda^{- \frac{2}{\gamma+1}}.
\end{equation}
\begin{lemma}
\label{lem: f positive}
$\forall c_{0}>0, \forall \lambda>0, \forall \omega\in [0,c_{0}^{-(2\gamma+1)}\lambda^{2}]$, we have $f(\lambda, \omega) \geq 1$.
\end{lemma}
\begin{proof}For $\omega\ge 0, \lambda>0$,
the inequality
$$
c_0 \ \omega \lambda^{- \frac{2}{\gamma+1}}\le \Big(\frac{\omega}{\lambda^{\frac{1}{\gamma +1}}}\Big)^{\frac{2\gamma}{2\gamma + 1}}
$$
is equivalent to 
$
c_{0}\omega^{\frac{1}{2\gamma+1}}\le \lambda^{\frac{2}{\gamma+1}-\frac{2\gamma}{(2\gamma+1)(\gamma+1)}}
=\lambda^{\frac2{2\gamma+1}}
$
\ie to 
$
c_{0}^{2\gamma+1} \omega\le \lambda^{2}.
$
\end{proof}
\section{Proof of Theorem~\ref{th:b-xn-indep}: the invariant case}
\label{sec:inariantcase}
\subsection{Reduction of Theorem~\ref{th:b-xn-indep} to a 
\texorpdfstring{$(n-1)$}{n-1} dimensional problem}
After a Fourier transform in the $x_n$ variable, Theorem~\ref{th:b-xn-indep} reduces to the following result.
\begin{theorem}
\label{th: resolvent estimate 1-D}
Assume~\eqref{eq: x2 indep}, \eqref{eq: x1 gamma} and define the operator acting on $L^2(\Mp)$
\begin{equation}\label{plamom}
P_{\lambda,\omega} =  -\Delta_{\Mp} - \omega + i \lambda b , \quad D(P_{\lambda,\omega}) = H^2(\Mp).
\end{equation}
Then, there exist $C>0$ and $\lambda_0 >0$ such that for all $u \in H^2(\Mp)$, for all $\lambda \geq \lambda_0$ and for all $\omega \leq \lambda^2$, we have 
\begin{align}
\label{eq: resolvent estimate 1-D}
\|P_{\lambda,\omega} u \|_{L^2(\Mp)} \geq C \lambda^{\frac{1}{\gamma+1}} \|u \|_{L^2(\Mp)} .
\end{align}
\end{theorem}
In this section, we only prove that Theorem~\ref{th: resolvent estimate 1-D} implies Theorem~\ref{th:b-xn-indep}. The proof of Theorem~\ref{th: resolvent estimate 1-D} needs more work and is completed in Section~\ref{sec:proof-th-reduit}.
\begin{proof}[Proof that Theorem~\ref{th: resolvent estimate 1-D} $\Rightarrow$ Theorem~\ref{th:b-xn-indep}]
We perform a Fourier transform in the variable $x_n \in \T^1$:
$$
u(x', x_n) = \sum_{k \in \Z} \hat{u}_k (x') e^{i k x_n} , \quad \text{with} \quad\hat{u}_k (x') =  \frac{1}{2\pi} \int_{0}^{2\pi} u(x',x_n) e^{- i k x_n} dx_n .
$$
Then, for $u \in H^2(M)$, we have  with $P_{\lambda}$ defined in \eqref{plambda} and $P_{\lambda,\omega}$ in \eqref{plamom},
$$
(P_\lambda u )(x',x_n) =  \sum_{k \in \Z} \left((-\Delta_{\Mp} + k^2 - \lambda^2 + i \lambda b)\hat{u}_k \right) (x') e^{i k x_n} = \sum_{k \in \Z} \left(P_{\lambda , \lambda^2-k^2}\hat{u}_k \right) (x') e^{i k x_n} ,
$$
as $b = b(x')$ does not depend on the $x_n$-variable. We hence obtain 
$$
 \|P_\lambda u\|_{L^2(M)}^2 = 2\pi \sum_{k \in \Z} \| P_{\lambda , \lambda^2-k^2}\hat{u}_k \|_{L^2(\Mp)}^2 .
$$
Finally, as a consequence of Theorem~\ref{th: resolvent estimate 1-D}, we have 
$\| P_{\lambda , \lambda^2-k^2} w \|_{L^2(\Mp)} \geq C \lambda^{\frac{1}{\gamma+1}} \|w\|_{L^2(\Mp)}$ where $C>0$ does not depend on $k$. This yields
$$
 \|P_\lambda u\|_{L^2(M)}^2 \geq 2\pi C^{2} \sum_{k \in \Z}\lambda^{\frac{2}{\gamma+1}} \| \hat{u}_k \|_{L^2(\Mp)}^2 = C^{2} \lambda^{\frac{2}{\gamma+1}} \| u\|_{L^2(M)}^2 ,
$$
which proves Theorem~\ref{th:b-xn-indep}.
\end{proof}
\subsection{Proof of Theorem~\ref{th: resolvent estimate 1-D}}
\label{sec:proof-th-reduit}
We now want to use Lemma~\ref{lem: P lambda omega} in a \nhd of $\{y'\} \times \T^1$ and to patch estimates together to complete  the proof of Theorem~\ref{th: resolvent estimate 1-D}.
\begin{proof}[Proof of Theorem~\ref{th: resolvent estimate 1-D}]
Let $\chi_0 \in \Cinfc(B(y',\eps_0);[0,1])$ such that $\chi_0=1$ in a  \nhd of $y'$ and set $\chi_1 = 1-\chi_0 \in \Cinf(\Mp)$. 
On the one hand, we have, with $P_{\lambda, \omega}$ given by
\eqref{plamom},
\begin{equation}
\label{eq estimee GCC++}
C\|P_{\lambda, \omega} \chi_1 u\|_{L^2(\Mp)} \geq  \lambda \|\chi_1 u\|_{L^2(\Mp)} ,
\end{equation}
{since, according to Assumption~\eqref{eq: x1 gamma}, $b$ is bounded from below on $\supp(\chi_1)$ and hence}
\begin{align*}
 \lambda \|\chi_1 u\|_{L^2(\Mp)}^2 &\leq C \lambda \langle b \chi_1 u ,\chi_1 u\rangle_{L^2(\Mp)}
 = C \Im \langle P_{\lambda, \omega} \chi_1 u,\chi_1 u \rangle_{L^2(\Mp)}\\
 &\leq C\|P_{\lambda, \omega} \chi_1 u\|_{L^2(\Mp)}\|\chi_1 u\|_{L^2(\Mp)} .
\end{align*}
On the other hand, we have
\begin{equation*}
\|P_{\lambda, \omega}  \chi_0 u\|_{L^2(\Mp)}^2 = \|P_{\lambda, \omega}  \chi_0 u\|_{L^2(B(y',\eps_0))}^2.
\end{equation*}
We write $x' = (x_1, \cdots x_{n-1})$. 
{According to Assumption~\eqref{eq: homo gamma}, we can extend (for instance by homogeneity in the variable $x'-y'$ outside $B_{\R^{n-1}}(y',\eps_0)$) the function $b$ from $B_{\R^{n-1}}(y',\eps_0)$ to the whole $\R^{n-1}$ as a measurable function $W$ satisfying
$$
C_1^{-1} |x'-y'|^{2\gamma} \leq W(x') \leq C_1 | x' -y'|^{2\gamma}  , \quad x' \in \R^{n-1} 
 ,\quad \text{and} \quad  b = W \text{ on }B(y',\eps_0) .
$$
}
Hence, we have 
$$
P_{\lambda , \omega} \chi_0 = \tilde{P}_{\lambda , \omega} \chi_0,
$$
where $\tilde{P}_{\lambda , \omega}$ is given by \eqref{ptildelom}.
We may then apply Lemma~\ref{lem: P lambda omega} to the operator $\tilde{P}_{\lambda,\omega}$ in $\R^{n-1}$. This yields, for some $C>0$,  
\begin{align}
\label{eq: fourier series++}
C\|P_{\lambda,\omega}  \chi_0 u\|_{L^2(\Mp)}^2 &= C\|\tilde{P}_{\lambda,\omega}  \chi_0 u\|_{L^2(\R^{n-1})}^2  \nonumber \\
& \ge   
 \lambda^{\frac{2}{\gamma+1}} \left( 1+ H(\omega)\left(\frac{\omega}{\lambda^{\frac{1}{\gamma +1}}}\right)^{\frac{2\gamma}{2\gamma + 1}} \right) \| \chi_0 u\|_{L^2(\R^{n-1})}^2 \nonumber \\
& =    
 \lambda^{\frac{2}{\gamma+1}} \left( 1+H(\omega)\left(\frac{\omega}{\lambda^{\frac{1}{\gamma +1}}}\right)^{\frac{2\gamma}{2\gamma + 1}} \right) \| \chi_0 u\|_{L^2(\Mp)}^2 .
\end{align}
We now want to estimate the remainder term
\begin{align}
\label{eq: commute tronc}
\|[P_{\lambda,\omega} ,\chi_1] u\|_{L^2(\Mp)}
= 
\|[P_{\lambda,\omega} ,\chi_0] u\|_{L^2(\Mp)}
& =\|[-\Delta_{\Mp} ,\chi_0] u\|_{L^2(\Mp)}
 \nonumber \\
&\leq \|(\Delta_{\Mp}\chi_0 ) u\|_{L^2(\Mp)}
 +2 \|\nabla_{x'}\chi_0 \cdot \nabla_{x'} u\|_{L^2(\Mp)} .
\end{align}
For this, we take $\psi = \psi(x')\in \Cinfc(B(0, \eps_0);[0,1])$ such that $\psi = 1$ on $\supp(\nabla_{x'}\chi_0)$ and $\psi = 0$ in a \nhd of $0$.
We compute
\begin{align}
\label{eq: estimate tronc}
\re \langle P_{\lambda, \omega} u , \psi^2 u\rangle_{L^2(\Mp)}
& = \re \langle (-\Delta_{\Mp} - \omega) u , \psi^2 u\rangle_{L^2(\Mp)} 
+\overbrace{ \re\langle i \lambda \psi^2 b u ,  u\rangle_{L^2(\Mp)}}^{=0} \nonumber \\
& = \re \langle -\Delta_{\Mp}u , \psi^2 u \rangle_{L^2(\Mp)}  - \omega \|\psi u \|_{L^2(\Mp)}^2  .
\end{align}
Moreover, we have
\begin{align*}
\re \langle -\Delta_{\R^{n-1}}u , \psi^2 u\rangle_{L^2(\Mp)} 
&= \langle \nabla_{x'}u , \psi^2\nabla_{x'} u\rangle_{L^2(\Mp)} 
+ \re \langle \nabla_{x'}u , u \nabla_{x'} \psi^2\rangle_{L^2(\Mp)} \\
&= \langle \nabla_{x'}u , \psi^2\nabla_{x'} u\rangle_{L^2(\Mp)} 
+ \sum_{j = 1}^{n-1}\re \langle i D_{x_j}u , u \d_{x_j} \psi^2\rangle_{L^2(\Mp)} \\
&= \langle \nabla_{x'}u , \psi^2\nabla_{x'} u\rangle_{L^2(\Mp)} 
- \sum_{j = 1}^{n-1}\frac{i}{2} \langle  [D_{x_j}, \d_{x_j} \psi^2]u , u \rangle_{L^2(\Mp)} \\
&= \langle \nabla_{x'}u , \psi^2\nabla_{x'} u\rangle_{L^2(\Mp)} 
-  \frac12\langle  (\Delta_{\R^{n-1}} \psi^2) u , u \rangle_{L^2(\Mp)} ,
\end{align*}
since the two operators $D_{x_j}$ and $\d_{x_j} \psi^2$ are selfadjoint.
We have thus 
\begin{multline*}
\norm{\psi\nabla_{x'} u}^{2}_{L^{2}(M')}
=\re \langle \bigl(-\Delta_{\R^{n-1}}-\omega+i\lambda b\bigr)u , \psi^2 u\rangle_{L^2(\Mp)} +
\overbrace{\re \langle \bigl(\omega-i\lambda b\bigr)u , \psi^2 u\rangle_{L^2(\Mp)} }^{=\omega\norm{\psi u}^{2}_{L^{2}(M')}}
\\+\frac12\langle  (\Delta_{\R^{n-1}} \psi^2) u , u \rangle_{L^2(\Mp)},
\end{multline*}
and consequently
we obtain
$$
\| \psi \nabla_{x'} u\|_{L^2(\Mp)}^2 \leq
\|P_{\lambda, \omega} u \|_{L^2(\Mp)} \| u\|_{L^2(\Mp)} + C_{1} \| u\|_{L^2(\Mp)}^2 
+  \omega \| \psi u \|_{L^2(\Mp)}^2.
$$
As a result, we can estimate the commutator of \eqref{eq: commute tronc} by
\begin{align}
\label{eq: comm}
\|[P_{\lambda, \omega} ,\chi_0] u\|_{L^2(\Mp)}^2\leq C_{2} \Big( 
\|P_{\lambda, \omega} u \|_{L^2(\Mp)} \| u\|_{L^2(\Mp)} + \| u\|_{L^2(\Mp)}^2 
+ \omega \| \psi u\|_{L^2(\Mp)}^2 \Big).
\end{align}
Now, we have
$
\norm{P_{\lambda, \omega}(\chi_{j}u)}_{L^{2}(M')}^{2}\le 
2\norm{[P_{\lambda, \omega},\chi_{j}]u}_{L^{2}(M')}^{2}+2\norm{\chi_{j}P_{\lambda, \omega}u}_{L^{2}(M')}^{2},
$
so that 
\begin{equation}
2\|P_{\lambda, \omega} u \|_{L^2(\Mp)}^2 
\ge
\|P_{\lambda, \omega} \chi_0 u \|_{L^2(\Mp)}^2 + \|P_\lambda \chi_1 u \|_{L^2(\Mp)}^2 
- 4 \|[P_{\lambda, \omega} ,\chi_0] u\|_{L^2(\Mp)}^2 ,
\end{equation}
which, combined with the estimates~
\eqref{eq estimee GCC++}, \eqref{eq: fourier series++}
 and~\eqref{eq: comm}, yields
\begin{align}
\label{eq: estimate -omega}
C_{3}\Bigl(\|P_{\lambda, \omega} u\|_{L^2(\Mp)}^2 + \|u\|_{L^2(\Mp)}^2\Bigr)
& \ge \lambda^{\frac{2}{\gamma+1}}\left( 1+ H(\omega)\left(\frac{\omega}{\lambda^{\frac{1}{\gamma +1}}}\right)^{\frac{2\gamma}{2\gamma + 1}} \right) \| \chi_0 u\|_{L^2(\Mp)}^2  \nonumber \\
&\quad  +  \lambda^2 \|\chi_1 u\|_{L^2(\Mp)}^2
 - c_1 \omega \| \psi u \|_{L^2(\Mp)}^2,
\end{align}
where  $c_{1}$ is a fixed positive constant.

In the r\'egime $\omega \leq 0$ (or, more generally, $\omega \leq \omega_0$ for any given $\omega_0$), this suffices to prove~\eqref{eq: resolvent estimate 1-D}. 
\vs
Let us now study the r\'egime $\omega \geq 0$. We notice that,
for $\omega \leq \lambda^2$, $\lambda\geq 1$ ,
\begin{align*}
\lambda^{\frac{2}{\gamma+1}}\Big( 1+ \Big(\frac{\omega}{\lambda^{\frac{1}{\gamma +1}}}\Big)^{\frac{2\gamma}{2\gamma + 1}} \Big)
& =\lambda^{\frac{2}{\gamma+1}}+
\lambda^{\frac{2}{\gamma+1}-\frac{2\gamma}{(2\gamma+1)(\gamma+1)}}
\omega^{\frac{2\gamma}{2\gamma+1}} \\
& \le\lambda^{\frac{2}{\gamma+1}}+
\lambda^{\frac{2}{\gamma+1}-\frac{2\gamma}{(2\gamma+1)(\gamma+1)}+\frac{4\gamma}{2\gamma+1}}
= \lambda^{\frac{2}{\gamma+1}} + \lambda^2 \leq 2\lambda^2,
\end{align*}
 and that, for all $v \in L^2(\Mp)$ we have
$$
\| \psi v\|_{L^2(\Mp)}^2 \le C_{4} \|  v\|_{L^2(\Mp)}^2=C_{4} \| (\chi_0 + \chi_1) v\|_{L^2(\Mp)}^2
\le 2C_{4}  \| \chi_0 v\|_{L^2(\Mp)}^2+2C_{4}\| \chi_1 v\|_{L^2(\Mp)}^2 .
$$
This, together with \eqref{eq: estimate -omega} then yields
\begin{align*}C_{5}
\Bigl(\|P_{\lambda, \omega} u\|_{L^2(\Mp)}^2 + \|u\|_{L^2(\Mp)}^2\Bigr)
&  \ge  \lambda^{\frac{2}{\gamma+1}}  f(\lambda, \omega) \| (\chi_0 + \chi_1) u \|_{L^2(\Mp)}^2 =  \lambda^{\frac{2}{\gamma+1}} f(\lambda, \omega) \| u \|_{L^2(\Mp)}^2,
\end{align*}
where $f(\lambda, \omega)$ is defined in~\eqref{eq: f(lambda,omega)}
with a fixed positive constant $c_{0}$.
According to Lemma~\ref{lem: f positive}, there exists $\lambda_0>0$ and 
\begin{equation}\label{delta0}
\delta =c_{0}^{-(2\gamma+1)}>0,
\end{equation}
 such that for all $\lambda \geq \lambda_0$ and $\omega \in [0, \delta \lambda^2]$, we have $f(\lambda, \omega) \geq 1$. As a consequence, \eqref{eq: resolvent estimate 1-D} is satisfied in this r\'egime.
Finally, suppose that $ \delta \lambda^2 \leq \omega \leq \lambda^2$, where $\delta$ is given by \eqref{delta0}. In this r\'egime, the estimate \eqref{eq: resolvent estimate 1-D} is a direct consequence of the usual (stronger) 1-microlocal estimate (see Lemma~\ref{lem: estimate 1 micro} below).
\end{proof}
\begin{lemma}
\label{lem: estimate 1 micro}
Let $\delta >0$, $y' \in \Mp$ and suppose that $b(y')=0$ and $b>0$ on $\Mp \setminus \{y'\}$. Then, there exists $\lambda_0>0$ and $C>0$ such that for all $\lambda\geq \lambda_0$, for all $\omega \in [\delta \lambda^2, \lambda^2]$, we have
$$
\|P_{\lambda, \omega} u\|_{L^2(\Mp)} \gtrsim \lambda  \| u \|_{L^2(\Mp)} .
$$
\end{lemma}
This lemma states the classical estimate associated to a ``one-microlocal'' propagation result in the presence of geometric control. We do not provide  a proof here since it is simpler than the proof of Lemma~\ref{lem: geom control} below and would follow exactly the same lines. The only additional difficulty with respect to the proof of Lemma~\ref{lem: geom control} is that the constants are uniform with respect to the parameter $\omega \in [\delta \lambda^2, \lambda^2]$ (whereas Lemma~\ref{lem: geom control} only tackles the case $\omega = \lambda^2$). It only requires a simple change of definition of the compact $K$ in the geometric definitions in the first part of the proof of Lemma~\ref{lem: geom control}.
\section{Proof of Theorem~\ref{th: resolvent estimate}: the non-invariant case}
\label{sec:proofmainth}
\subsection{Proof of a geometric control lemma}
\label{s:geometriccontrol}
In this section, we prove the following lemma. All definitions and tools of geometry and pseudodifferential calculus used in the proof are introduced in Appendices~\ref{s:geometry} and~\ref{s:pseudodiff} respectively.
{\begin{lemma}
\label{lem: geom control}
Assume that $b \in L^\infty(M; \R^+)$ and recall that $\omega_b$ is defined in~\eqref{eq:defomega}. Take a non-negative function $\alpha \in S^0_{0,0}(T^*M)$. Assume that for all $\rho \in \supp(\alpha) \cap S^*M$ there exists $t \in \R$ such that $\phi_{t}(\rho) \in T^*\omega_b$. Then, there exist $C,\lambda_0 >0$ such that for all $\lambda \geq \lambda_0$, we have
\begin{align*}
\lambda \|\Op \bigl(\alpha(x, \frac{\xi}{\lambda})\bigr) u\|_{L^2(M)} \leq C \|P_\lambda \Op \bigl(\alpha(x, \frac{\xi}{\lambda})\bigr) u \|_{L^2(M)} + C \|u \|_{L^2(M)} .
\end{align*}
\end{lemma}
}
 {This Lemma states a ``one-microlocal'' estimate in the presence of a partial geometric control, which is adapted to our needs.} We give a proof here to check that no smoothness is required on $b$. Moreover, the proof below uses multiplier estimates and is hence of constructive type. 
Note that if $\omega_b$ satisfies (GCC), then the assumption of the lemma is satisfied by $\alpha = 1$ and the lemma yields the optimal estimate
\begin{align*}
\lambda \| u\|_{L^2(M)} \leq C \|P_\lambda  u \|_{L^2(M)} .
\end{align*}
This estimate is equivalent to the uniform (and hence exponential) decay of the associated problem~\eqref{eq: stabilization} (see for instance~\cite{AL:13} and the references therein).
\begin{proof}
The proof is divided in several steps.
\paragraph{\bfseries Some geometric facts.}
 Set $K := \supp(\alpha) \cap S^*M \subset T^*M$ and, for $\rho \in K$ denote $t_\rho \in \R$ a time such that $\phi_{t_\rho}(\rho) \in T^*\omega_b$. The compact set $K$ is hence such that for all $\rho \in K$, $\rho \in \phi_{- t_\rho}(T^*\omega_b)$, i.e.
$$K \subset \bigcup_{t \in\R}\phi_{-t}( T^*(\omega_b)).$$
{Moreover, for $\ell \in \N^*$, we let $\omega^\ell$
be a family of open sets such that $\ovl{\omega}^\ell \subset \omega_b$, $\ovl{\omega}^\ell \subset \omega^{\ell+1}$, and $\bigcup_{\ell \in \N^*}\omega^\ell = \omega_b$. For any $t \in \R$, we have $\phi_{-t}(T^*\omega_b) = \bigcup_{\ell \in \N^*}\phi_{-t}(T^*\omega^\ell)$, where $\phi_{-t}(T^*\omega^\ell)$ are open subsets of $T^*M$. This reads}
$$K \subset \bigcup_{t \in \R}\bigcup_{\ell \in \N^*}\phi_{-t}(T^*\omega^\ell).$$
Since $K$ is compact, we may hence extract a finite open cover of $K$, that is a finite number of times $t_j$, $j \in \{1, \cdots, J\}$ and of $\omega^\ell$, $\ell \in \{1, \cdots, L\}$ such that 
$$
K \subset \bigcup_{ j  =1}^J  \bigcup_{\ell = 1}^L \phi_{-t_j}(T^*\omega^\ell) .
$$
We then remark that $\omega^\ell \subset \omega^{\ell+1}$, so that, for all $t \in \R$, $ \bigcup_{\ell = 1}^L \phi_{-t}(T^*\omega^\ell)  = \phi_{-t}(T^*\omega^L)$. This finally yields 
$$
K \subset \bigcup_{ j =1}^J\phi_{-t_j}( T^*\omega^L), \quad \text{ with } \omega^L \text{ open such that } \ovl{\omega}^L \subset \omega_b  . 
$$
As $\bigcup_{ j =1}^J\phi_{-t_j}( T^*\omega^L)$ is open, we also have
$$K_\gamma := \supp(\alpha) \cap \{(x, \xi) \in T^*M , 1-\gamma \leq |\xi|_x \leq 1+\gamma\}\subset \bigcup_{ j =1}^J\phi_{-t_j}( T^*\omega^L)
$$
 for $\gamma \in (0,1)$ sufficiently small (fixed from now on).
\vs
Note that, at this point, we have in particular proved that the assumption of the Lemma, satisfied by the set $\omega_b$ in arbitrarily large time (depending on the point $\rho$), is actually satisfied by the smaller set $\omega^L$ in a uniform time (namely $\max_{1\leq i,j\leq J}|t_j -t_i|$).
\vs 
{Since $\ovl{\omega}^L$ is a compact subset of $M$ such that 
$$\ovl{\omega}^L \subset \omega_b=\bigcup \left\{U\subset M, U \text{ open}, \essinf_U(b) >0 \right\},$$
we can extract a finite cover of $\ovl{\omega}^L$:
$$\ovl{\omega}^L \subset\bigcup_{k=1}^\kappa \left\{U_k \subset M, U_k \text{ open}, \essinf_{U_k}(b) >0 \right\} .$$
Hence, we have
$$
b \geq \min\{ \essinf_{U_k}(b) ,1\leq  k \leq \kappa\} >0 , \quad \text{a.e. on } \ovl{\omega}^L . 
$$
Now we define $\tilde{b} = \eps \chi$ with $\chi \in \Cinfc(\omega_b; [0,1])$ such that $\chi =1$ on $\omega^L$ and $\eps = \min\{ \essinf_{U_k}(b) ,1\leq  k \leq \kappa\}$ so that we have
$$\tilde{b}\in \Cinf(M), \quad \omega^L \subset \{\tilde{b}>0\}, \quad 0\leq \tilde{b}  \leq b . 
$$ }
This yields in particular
 \begin{align}
\label{geomeq:1-chi0-b}
\lambda \langle \tilde{b} u,u \rangle_{L^2(M)} \lesssim \lambda \langle b u ,u\rangle_{L^2(M)} 
= \Im\langle P_\lambda u, u\rangle_{L^2(M)} 
\leq \| P_\lambda u\|_{L^2(M)} \| u \|_{L^2(M)} .
\end{align}
\paragraph{\bfseries Definition of the multipliers.}
Denoting by $O_j  = \phi_{-t_j}( T^* \omega^L) $, the sets $(O_j)_{j \in \{1, \cdots, J\}}$ form a finite open cover of $K_\gamma$ such that $\phi_{t_j}(O_j) \subset T^*\omega^L$. We denote $(\chi_j)_{j \in \{1, \cdots, J\}}$ a partition of unity of $K_\gamma$ subordinated to $(O_j)_{j \in \{1, \cdots, J\}}$. 
We  set $$\psi_j (x, \xi) = - \int_0^{t_j}\chi_j \circ \phi_\tau (x, \xi) d\tau.$$ Taking $\chi \in \Cinfc(\R ;[0,1])$ such that $\chi =1$ a \nhd of $1$ and $\supp\chi \subset (1-\gamma, 1+\gamma)$, we now define the multipliers 
$$
m_j  (x, \frac{\xi}{\lambda})= \chi(\frac{|\xi|_x^2}{\lambda^2}) \psi_j (x , \frac{\xi}{\lambda}), \quad j \in \{1, \cdots ,J\}.
$$
To $m_j$, we associate an operator $\Op (m_j)$, bounded 
on $L^2(M)$ (see Appendix~\ref{s:pseudodiff}).
  \paragraph{\bfseries Estimate in the propagative region.}
We have,  with $P_{\lambda}$ defined  in \eqref{plambda},
\begin{align}
\label{geomeq: multiplier lambda}
2 \Re \langle P_\lambda  u , i \Op (m_j) u \rangle_{L^2(M)} &= \left<(\left[ (-\Delta -\lambda^2) ,i \Op (m_j) \right] u , u \right>_{L^2(M)} + 2 \Re \left< \lambda b u , \Op (m_j) u \right>_{L^2(M)} \nonumber \\
&= \re \left<\Op \left(\left\{ |\xi|_x^2-\lambda^2 , m_j \right\}\right) u , u \right>_{L^2(M)} + O(1)\|u\|_{L^2(M)}^2 \nonumber \\
& \qquad + 2 \lambda \Re \left< b u , \Op (m_j) u \right>_{L^2(M)} ,
\end{align}
according to symbolic calculus.
Moreover, we have 
\begin{align*}
\left\{ |\xi|_x^2-\lambda^2 , m_j \right\} &= H_p m_j = \chi(\frac{|\xi|_x^2}{\lambda^2}) H_p\left( \psi_j (x,  \frac{\xi}{\lambda}) \right)\\
&= - \chi(\frac{|\xi|_x^2}{\lambda^2}) \int_0^{t_j} H_p\left( \chi_j \circ \phi_\tau (x, \frac{\xi}{\lambda}) \right) d\tau\\
&= - \chi(\frac{|\xi|_x^2}{\lambda^2}) \int_0^{t_j} \lambda \frac{d}{d\tau}\left( \chi_j \circ \phi_\tau (x, \frac{\xi}{\lambda}) \right) d\tau\\
&= \lambda \chi(\frac{|\xi|_x^2}{\lambda^2}) \left(\chi_j (x, \frac{\xi}{\lambda}) - \chi_j  \circ \phi_{t_j}(x, \frac{\xi}{\lambda}) \right),
\end{align*}
as $\dot{\phi}_\tau= H_p(\phi_\tau)$.
\vs
Coming back to~\eqref{geomeq: multiplier lambda}, and using the boundedness of $\Op (m_j)$, we obtain
\begin{multline}
\label{geomeq:propagative-1-micro}
\|P_\lambda  u \|_{L^2(M)} \| u \|_{L^2(M)} + \| u \|_{L^2(M)}^2 + \lambda \|b u \|_{L^2(M)} \|u  \|_{L^2(M)}
\\
\gtrsim \lambda \left< \Op \bigl(\chi(\frac{|\xi|_x^2}{\lambda^2}) \chi_j (x, \frac{\xi}{\lambda}) \bigr) u, u \right>_{L^2(M)}
- \lambda \left< \Op \left(\chi(\frac{|\xi|_x^2}{\lambda^2})  \chi_j \circ \phi_{t_j}(x, \frac{\xi}{\lambda}) \right) u, u \right>_{L^2(M)} .
\end{multline}
Next, by construction, the function $(x,\xi) \mapsto \chi(|\xi|_x)  \chi_j \circ \phi_{t_j}(x, \xi)$ is supported in $T^*\omega^L$ where $\tilde{b}>0$, so that we have 
$$
\chi(\frac{|\xi|_x^2}{\lambda^2})  \chi_j \circ \phi_{t_j}(x, \frac{\xi}{\lambda}) \lesssim \tilde{b}(x)
$$
uniformly on $T^*M$.
According to the sharp G{\aa}rding inequality, this yields
\begin{align*}
\lambda \left< \Op \left(\chi(\frac{|\xi|_x^2}{\lambda^2})  \chi_j \circ \phi_{t_j}(x, \frac{\xi}{\lambda}) \right) u, u \right>_{L^2(M)}&
 \lesssim \lambda \left<\tilde{b}u,u\right>_{L^2(M)} + \| u\|_{L^2(M)}^2 \\
 & \lesssim \| P_\lambda u\|_{L^2(M)} \| u \|_{L^2(M)} + \| u\|_{L^2(M)}^2 , 
\end{align*}
when using~\eqref{geomeq:1-chi0-b}. Combined with~\eqref{geomeq:propagative-1-micro} and 
$$
\lambda \|b u \|_{L^2(M)}^2 \lesssim \lambda \left< b u ,u \right>_{L^2(M)} 
\leq \| P_\lambda u\|_{L^2(M)} \| u \|_{L^2(M)} ,
$$
this implies, for all $j \in \{1, \cdots, J\}$, 
\begin{multline}
\label{geomeq:propagative-1-micro-bis}
\|P_\lambda  u \|_{L^2(M)} \| u \|_{L^2(M)} + \| u \|_{L^2(M)}^2 + \lambda^\frac12 \|P_\lambda  u \|_{L^2(M)}^\frac12 \|u  \|_{L^2(M)}^\frac32
\\
\gtrsim \lambda \left< \Op \left(\chi(\frac{|\xi|_x^2}{\lambda^2}) \chi_j(x, \frac{\xi}{\lambda}) \right) u, u \right>_{L^2(M)} .
\end{multline}
\paragraph{\bfseries Estimate in the elliptic region.}
Next, we estimate $$\Op (1- \chi(\frac{|\xi|_x^2}{\lambda^2})) u.$$ 
Take $\theta\in \Cinf(\R^+ ;[-1,1])$ such that $\theta = 0$ in a \nhd of $1$, $\theta = -1$ in a \nhd of $0$ and $\theta = 1$ ouside of a \nhd of $[0,1]$, so that $|\theta| = 1- \chi$.
We compute 
\begin{multline}
\label{eq: remainder elliptic}
\re\left< P_\lambda u , \Op \left(\theta(\frac{|\xi|_x^2}{\lambda^2})\right)u\right>_{L^2(M)}
=\Re\left< (-\Delta_g - \lambda^2)u , \Op \left(\theta(\frac{|\xi|_x^2}{\lambda^2})\right)u\right>_{L^2(M)}
\\ + \re\left< i \lambda b  u , \Op\left(\theta(\frac{|\xi|_x^2}{\lambda^2}) \right) u\right>_{L^2(M)}  .
\end{multline}
Then, using symbolic calculus, we have 
\begin{align*}
&\left( \Op \left(\theta(\frac{|\xi|_x^2}{\lambda^2})\right)\right)^* (-\Delta_g - \lambda^2) 
\\& \hs\hs= \Op\big(\theta(\frac{|\xi|_x^2}{\lambda^2})  (|\xi|_x^{2}-\lambda^{2})\big) + O_{\calL(L^2(M))}(\lambda) + O_{\calL(H^1 (M);L^2(M))}(1)\\
 & \hs\hs= \Op \big( |\theta(\frac{|\xi|_x^2}{\lambda^2}) | (|\xi|_x^{2} +\lambda^{2})\big) + O_{\calL(L^2(M))}(\lambda) + O_{\calL(H^1 (M);L^2(M))}(1),
\end{align*}
so that the sharp G{\aa}rding inequality yields 
\begin{align}
\label{e:gardingH1L2rem}
\Re \left< (-\Delta_g - \lambda^2)u , \Op \left(\theta(\frac{|\xi|_x^2}{\lambda^2}) \right)u\right>_{L^2(M)} & \gtrsim \left< \Op \big( |\theta(\frac{|\xi|_x^2}{\lambda^2}) | (|\xi|_x^{2} +\lambda^{2})\big)u , u \right>_{L^2(M)}  - O(\lambda) \|u  \|_{L^2(M)}^2 \nonumber\\
& \qquad - O(1) \|u  \|_{H^\frac12(M)}^2 - O(1)  \|u  \|_{H^1 (M)} \|u  \|_{L^2 (M)} \nonumber\\
& \gtrsim \lambda^{2} \left< \Op \big(1- \chi(\frac{|\xi|_x^2}{\lambda^2}) \big) u , u \right>_{L^2(M)}  - O(\lambda) \|u  \|_{L^2(M)}^2 \nonumber\\
& \qquad - O(1)  \|u  \|_{H^1 (M)} \|u  \|_{L^2 (M)}, 
\end{align}
since $\|u  \|_{H^\frac12 (M)}^2 \leq \|u  \|_{H^1 (M)} \|u  \|_{L^2 (M)}$.
To estimate $\|u  \|_{H^1 (M)} \|u  \|_{L^2 (M)}$, we simply write
\begin{align*}
\|P_\lambda  u \|_{L^2(M)} \| u \|_{L^2(M)} \geq \Re \left< P_\lambda  u,  u \right>_{L^2(M)} 
&  = \left< (-\Delta_g - \lambda^2)u,  u \right>_{L^2(M)} \\
&  = \|\nabla_g u\|_{L^2(M)}^2 - \lambda^2\| u \|_{L^2(M)}^2 ,
\end{align*}
so that
$$
 \|u  \|_{H^1 (M)} \|u  \|_{L^2 (M)} \lesssim \lambda \| u \|_{L^2(M)}^2 +  \|P_\lambda  u \|_{L^2(M)}^\frac12 \| u \|_{L^2(M)}^\frac32 .
$$ 
Coming back to~\eqref{e:gardingH1L2rem}, this yields 
 \begin{multline*}
\left< (-\Delta_g - \lambda^2)u , \Op \left( \theta(\frac{|\xi|_x^2}{\lambda^2}) \right)u\right>_{L^2(M)} + \lambda \|u  \|_{L^2(M)}^2 + \|P_\lambda  u \|_{L^2(M)} \| u \|_{L^2(M)} 
\\
\gtrsim \lambda^{2} \left< \Op \big(1- \chi(\frac{|\xi|_x^2}{\lambda^2}) \big) u , u \right>_{L^2(M)}  . 
\end{multline*}
Moreover, as above, we have
\begin{align*}
\left|\re\left< i \lambda b  u , \Op \left( \theta(\frac{|\xi|_x^2}{\lambda^2}) \right) u\right> \right| 
&\lesssim \lambda \|b  u\|_{L^2(M)} \| u\|_{L^2(M)} \lesssim \lambda^\frac12 \|P_\lambda  u \|_{L^2(M)}^\frac12 \|u  \|_{L^2(M)}^\frac32 \\
&\lesssim  \|P_\lambda  u \|_{L^2(M)}\|u  \|_{L^2(M)} + \lambda \|u  \|_{L^2(M)}^2   . 
\end{align*}
Coming back to~\eqref{eq: remainder elliptic}, this implies
\begin{align*}
\|P_\lambda  u \|_{L^2(M)} \| u \|_{L^2(M)} + \lambda \| u \|_{L^2(M)}^2 
\gtrsim \lambda^{2} \left< \Op \big(1- \chi(\frac{|\xi|_x^2}{\lambda^2}) \big) u , u \right>_{L^2(M)}  . 
\end{align*}
Dividing this estimate by $\lambda$, we finally obtain
\begin{align}
\label{geomeq:elliptic-1-micro}
  \lambda^{-1} \|P_\lambda  u \|_{L^2(M)} \| u \|_{L^2(M)}  +  \| u \|_{L^2(M)}^2 
\gtrsim \lambda \left<\Op  \big(1- \chi(\frac{|\xi|_x^2}{\lambda^2}) \big) u , u \right>_{L^2(M)} .
\end{align}
\paragraph{\bfseries Patching estimates together.}
Finally, according to the construction of $\chi_j , \chi$ and denoting 
$$
\Omega := \bigg\{(x,\xi) \in T^*M , \quad  \bigl(1- \chi(|\xi|_x^2)\bigr) +  \chi(|\xi|_x^2) \sum_{j=1}^J \chi_j (x, \xi ) > 0 \bigg\},
$$
we have $\supp(\alpha) \subset \Omega$. Let $\beta \in \Cinf(T^*M)$ be a function which is homogeneous of degree zero in each fiber for $|\xi|_x$ large, such that $\supp(\beta) \subset \Omega$ and $\beta = 1$ on a \nhd of $\supp(\alpha)$. We have 
$$
 \bigl(1- \chi(|\xi|_x^2)\bigr) + \chi(|\xi|_x^2) \sum_{j=1}^J \chi_j (x, \xi ) \gtrsim \beta^2 (x, \xi),
$$
uniformly on $T^*M$.
As a consequence, according to the sharp G{\aa}rding inequality, we have
\begin{multline*}
\lambda \bigg< \Op \Big(\chi(\frac{|\xi|_x^2}{\lambda^2}) \sum_{j=1}^J \chi_j (x, \frac{\xi}{\lambda})  + \bigl(1- \chi(\frac{|\xi|_x^2}{\lambda^2})\bigr) \Big) u, u \bigg>_{L^2(M)} 
\\\gtrsim\lambda \left\| \Op \bigl( \beta(x, \frac{\xi}{\lambda}) \bigr) u \right\|_{L^2(M)}^2- C \| u \|_{L^2(M)}^2 .
\end{multline*}
Using inequalities~\eqref{geomeq:propagative-1-micro-bis} and~\eqref{geomeq:elliptic-1-micro} yields 
\begin{align*}
\lambda \left\| \Op \bigl( \beta(x, \frac{\xi}{\lambda}) \bigr) u \right\|_{L^2(M)}^2 & \lesssim
\|P_\lambda  u \|_{L^2(M)} \| u \|_{L^2(M)} + \| u \|_{L^2(M)}^2 + \lambda^\frac12 \|P_\lambda  u \|_{L^2(M)}^\frac12 \|u  \|_{L^2(M)}^\frac32.
\end{align*}
Applying this estimate to $u$ replaced by $$A u := \Op \big( \alpha(x, \frac{\xi}{\lambda}) \big) u$$ and using that 
$$ \Op \bigl( \beta(x, \frac{\xi}{\lambda}) \bigr) \Op \bigl( \alpha(x, \frac{\xi}{\lambda})\bigr)  = \Op \bigl( \alpha(x, \frac{\xi}{\lambda}) \bigr) + O_{\calL(L^2(M))}(\lambda^{-1}),$$ 
according to proposition~\ref{p:suppdisj}, we obtain
\begin{align*}
\lambda \left\| A u \right\|_{L^2(M)}^2 & \lesssim
\|P_\lambda A  u \|_{L^2(M)} \| A u \|_{L^2(M)} + \| A u \|_{L^2(M)}^2 \\
& \hskip55pt+ \lambda^\frac12 \|P_\lambda A  u \|_{L^2(M)}^\frac12 \| A u  \|_{L^2(M)}^\frac32 + O(\lambda^{-1}) \| u  \|_{L^2(M)}^2 \\
 & \lesssim
(1 + \eps^{-1})\|P_\lambda A u \|_{L^2(M)} \| A u \|_{L^2(M)} +(1  +\eps \lambda ) \| A u \|_{L^2(M)}^2 + O(\lambda^{-1}) \| u  \|_{L^2(M)}^2 
\end{align*}
for all $\eps>0$. Choosing $\eps$ sufficiently small, yields
\begin{align*}
\lambda \left\| A u \right\|_{L^2(M)}^2 & \lesssim
\|P_\lambda A  u \|_{L^2(M)} \| A u \|_{L^2(M)} + \| A u \|_{L^2(M)}^2 +  O(\lambda^{-1}) \| u  \|_{L^2(M)}^2 \\
 & \lesssim
(\eps \lambda)^{-1}\|P_\lambda A u \|_{L^2(M)}^2 + (1 + \eps \lambda) \| A u \|_{L^2(M)}^2 + O(\lambda^{-1}) \| u  \|_{L^2(M)}^2 .
\end{align*}
for all $\eps>0$. Taking $\eps$ sufficiently small, and then $\lambda \geq \lambda_0$ for $\lambda_0$ large enough  provides the proof of the lemma.
\end{proof}
\subsection{End of the proof of Theorem~\ref{th: resolvent estimate}}
\label{sec:patchingestimates}
{
Next, we want to patch estimates together to complete  the proof of Theorem~\ref{th: resolvent estimate}.
\begin{proof}[Proof of Theorem~\ref{th: resolvent estimate}]
Let $U_{\eps_0} = B(y',\eps_0) \times \T^{n''} \subset M$ and $\chi_0 \in \Cinfc(B(y',\eps_0 );[0,1])$ such that $\chi_0=1$ in the \nhd of $y'$. We shall also write $\chi_0$ instead of $\chi_0 \otimes 1 \in \Cinfc(U_{\eps_0})$, and instead of $\chi_0 \otimes 1 \in \Cinf(M)$ where this function has been extended by $0$ in $M \setminus U_{\eps_0}$. We denote by $\chi_1 = 1-\chi_0 \in \Cinf(M)$. 

On the one hand, as a consequence of Lemma~\ref{lem: geom control} and Assumption~\eqref{eq: GCC outside}, we have for $\lambda \geq \lambda_0$,
\begin{equation}
\label{eq estimee GCC}
O(\lambda^{-1}) \| u\|_{L^2(M)} +\|P_\lambda \chi_1 u\|_{L^2(M)} \gtrsim  \lambda \|\chi_1 u\|_{L^2(M)}.
\end{equation}
On the other hand, we have
\begin{equation*}
\|P_\lambda  \chi_0 u\|_{L^2(M)}^2 = \|P_\lambda  \chi_0 u\|_{L^2(U_{\eps_0})}^2.
\end{equation*}
\vs\no
{\bf Notation: }in $U_{\eps_0}$, we note the coordinates $(x', x'')$ with $x' = (x_1' , \cdots ,x'_{n'})\in \R^{n'}$ and $x'' = (x_1'' , \cdots ,x''_{n''})\in \R^{n''}$.
\vs\no
{According to Assumption~\eqref{eq: homo gamma++}, we extend (for instance by homogeneity in the variable $x'-y'$ outside $B_{\R^{n'}}(y',\eps_0)$) the function $b= b(x')$ from $B_{\R^{n'}}(y',\eps_0)$ to the whole $\R^{n'}$ as a measurable function $W$ satisfying
$$
C_1^{-1} |x'-y'|^{2\gamma} \leq W(x') \leq C_1 | x' -y'|^{2\gamma}  , \quad x' \in \R^{n'} 
 ,\quad \text{and} \quad  b = W \text{ on }B(y',\eps_0) .
$$}
Hence, we have 
$$
P_\lambda  \chi_0 = ( -\Delta_{\R^{n'}} -\Delta_{\T^{n''}} - \lambda^2 + i \lambda  W(x'))\circ \chi_0
$$
As a consequence, we have
\begin{align*}
\|P_\lambda  \chi_0 u\|_{L^2(M)}^2 
& = \|(-\Delta_{\R^{n'}}-\Delta_{\T^{n''}} - \lambda^2 + i \lambda W (x'))( \chi_0 u)\|_{L^2(\R^{n'} \times \T^{n''})}^2  \\
& = \sum_{k \in \Z} \|(-\Delta_{\R^{n'}} + |k|^2 - \lambda^2 + i \lambda W (x'))( \chi_0(x') \hat{u}_k)\|_{L^2(\R^{n'})}^2 ,
\end{align*}
where we have denoted by $\hat{u}_k(x') =\frac{1}{(2\pi)^{n''}}\int_{\T^{n''}} u(x',x'') e^{- i k \cdot x''}dx''$, $k \in \Z^{n''}$ the partial Fourier transform in the periodic directions.
\par
Next, we apply  Lemma~\ref{lem: P lambda omega} to the operator $-\Delta_{\R^{n'}}  - (\lambda^2-|k|^2) + i \lambda W (x')$ in $\R^{n'}$ for any $k \in \Z^{n''}$. This yields, for some $C>0$,  
\begin{multline}
\label{eq: fourier series}
C\|P_\lambda  \chi_0 u\|_{L^2(M)}^2 
 \ge  \lambda^{\frac{2}{\gamma+1}}\sum_{\lambda^2-|k|^2 \leq 0} \| \chi_0(x') \hat{u}_k\|_{L^2(\R^{n'})}^2  \\
+  \lambda^{\frac{2}{\gamma+1}}\sum_{\lambda^2-|k|^2 > 0}\left( 1+ \left(\frac{\lambda^2-|k|^2}{\lambda^{\frac{1}{\gamma +1}}}\right)^{\frac{2\gamma}{2\gamma + 1}} \right) \| \chi_0(x') \hat{u}_k\|_{L^2(\R^{n'})}^2.
\end{multline}
We now want to estimate the remainder term
\begin{multline}
\label{eq: commute tronc+}
\|[P_\lambda ,\chi_1] u\|_{L^2(M)}
= 
\|[P_\lambda ,\chi_0] u\|_{L^2(M)}
 =\|[-\Delta_{\R^{n'}} ,\chi_0(x')] u\|_{L^2(\R^{n'} \times \T^{n''})}
 \\
\leq \|(\Delta_{\R^{n'}}\chi_0 ) u\|_{L^2(\R^{n'} \times \T^{n''})}
 +2 \|\nabla_{x'}\chi_0 \cdot \nabla_{x'} u\|_{L^2(\R^{n'} \times \T^{n''})}
\end{multline}
To estimate this remainder, we take now $\psi = \psi(x')= \psi \otimes 1 \in \Cinfc(U_{\eps_0})$ such that $\psi = 1$ on $\supp(\nabla_{x'}\chi)$.
We compute
\begin{align*}
\re \langle P_\lambda u , \psi^2 u\rangle_{L^2(M)}
 &= \re \langle(-\Delta - \lambda^2) u , \psi^2 u\rangle_{L^2(M)} 
+ \re\langle i \lambda \psi^2 b u ,  u\rangle_{L^2(M)} 
\\&= \re \langle-\Delta_{\R^{n'}}u , \psi^2 u \rangle_{L^2(\R^{n'} \times \T^{n''})} 
+ \langle \psi^2 \nabla_{x''}u , \nabla_{x''} u\rangle_{L^2(\R^{n'} \times \T^{n''})} \\
&\hskip150pt- \lambda^2 \langle \psi^2 u ,  u\rangle_{L^2(\R^{n'} \times \T^{n''})}.
\end{align*}
Moreover, we have
\begin{align*}
&\re \langle -\Delta_{\R^{n'}}u , \psi^2 u \rangle_{L^2(\R^{n'} \times \T^{n''})} 
\\&= \langle\nabla_{x'}u , \psi^2\nabla_{x'} u\rangle_{L^2(\R^{n'} \times \T^{n''})} 
+ \re \langle\nabla_{x'}u , u \nabla_{x'} \psi^2\rangle_{L^2(\R^{n'} \times \T^{n''})} \\
&= \langle\nabla_{x'}u , \psi^2\nabla_{x'} u\rangle_{L^2(\R^{n'} \times \T^{n''})} 
+ \sum_{j = 1}^{n'}\re\langle i D_{x_j}u , u \d_{x_j} \psi^2\rangle_{L^2(\R^{n'} \times \T^{n''})} \\
&= \langle\nabla_{x'}u , \psi^2\nabla_{x'} u\rangle_{L^2(\R^{n'} \times \T^{n''})} 
+ \sum_{j = 1}^{n'}\frac{i}{2} \langle [D_{x_j}, \d_{x_j} \psi^2]u , u \rangle_{L^2(\R^{n'} \times \T^{n''})} \\
&= \langle\nabla_{x'}u , \psi^2\nabla_{x'} u\rangle_{L^2(\R^{n'} \times \T^{n''})} 
+  \frac12\langle (\Delta_{\R^{n'}} \psi^2) u , u \rangle_{L^2(\R^{n'} \times \T^{n''})} ,
\end{align*}
when using that the two operators $D_{x_j}$ and $\d_{x_j} \psi^2$ are selfadjoint.
Now, we write 
$$
\langle \psi^2 \nabla_{x''}u , \nabla_{x''}u\rangle_{L^2(\R^{n'} \times \T^{n''})} 
- \lambda^2 \langle\psi^2 u ,  u\rangle_{L^2(\R^{n'} \times \T^{n''})} 
= \sum_{k \in \Z}(  |k|^2 - \lambda^2 )\| \psi(x') \hat{u}_k\|_{L^2(\R^{n'})}^2 ,
$$ 
and come back to \eqref{eq: estimate tronc} to obtain
\begin{multline*}
\| \psi \nabla_{x'} u\|_{L^2(\R^{n'} \times \T^{n''})}^2 \\\leq
\|P_\lambda u \|_{L^2(M)} \| u\|_{L^2(M)} + C \| u\|_{L^2(M)}^2 
+  \sum_{ \lambda^2- |k|^2 > 0}(  \lambda^2- |k|^2 )\| \psi(x') \hat{u}_k\|_{L^2(\R^{n'})}^2.
\end{multline*}
As a consequence, we estimate the commutator of \eqref{eq: commute tronc} by
\begin{multline}
\label{eq: comm+}
\|[P_\lambda ,\chi_0] u\|_{L^2(M)}^2\leq c_0 \Big( 
\|P_\lambda u \|_{L^2(M)} \| u\|_{L^2(M)} + \| u\|_{L^2(M)}^2 \\
+  \sum_{ \lambda^2- |k|^2 > 0}(  \lambda^2- |k|^2 )\| \psi(x') \hat{u}_k\|_{L^2(\R^{n'})}^2 \Big).
\end{multline}
Now, we have
\begin{equation}
\|P_\lambda u \|_{L^2(M)}^2 
\gtrsim \|P_\lambda \chi_0 u \|_{L^2(M)}^2 + \|P_\lambda \chi_1 u \|_{L^2(M)}^2 
- C \|[P_\lambda ,\chi_0] u\|_{L^2(M)}^2 ,
\end{equation}
which, combined with the estimates,~\eqref{eq estimee GCC}~\eqref{eq: fourier series} and~\eqref{eq: comm}, yields
\begin{align*}
\|P_\lambda u\|_{L^2(M)}^2 + \|u\|_{L^2(M)}^2
&  \gtrsim \lambda^2 \|\chi_1 u\|_{L^2(M)}^2 +\lambda^{\frac{2}{\gamma+1}}\sum_{\lambda^2-|k|^2 \leq 0} \| \chi_0(x') \hat{u}_k\|_{L^2(\R^{n'})}^2  \\
& +  \lambda^{\frac{2}{\gamma+1}}\sum_{\lambda^2-|k|^2 > 0}\left( 1+ \left(\frac{\lambda^2-|k|^2}{\lambda^{\frac{1}{\gamma +1}}}\right)^{\frac{2\gamma}{2\gamma + 1}} \right) \| \chi_0(x') \hat{u}_k\|_{L^2(\R^{n'})}^2\\
& - c_0 \sum_{ \lambda^2- |k|^2 > 0}(  \lambda^2- |k|^2 )\| \psi(x') \hat{u}_k\|_{L^2(\R^{n'})}^2 .
\end{align*}
Next, we remark that the cutoff functions are chosen so that, for all $v \in L^2(\Mp)$ we have
$$
\| \psi v\|_{L^2(\Mp)}^2 \lesssim \| (\chi_0 + \chi_1) v\|_{L^2(\Mp)}^2
\lesssim  \| \chi_0 v\|_{L^2(\Mp)}^2+\| \chi_1 v\|_{L^2(\Mp)}^2 .
$$
We hence obtain
\begin{align*}
\|P_\lambda u\|_{L^2(M)}^2 + \|u\|_{L^2(M)}^2
&  \gtrsim \lambda^2 \|\chi_1 u\|_{L^2(M)}^2 + \lambda^{\frac{2}{\gamma+1}}\sum_{\lambda^2-|k|^2 \leq 0} \| \chi_0 \hat{u}_k\|_{L^2(\Mp)}^2 +\| \chi_1 \hat{u}_k\|_{L^2(\Mp)}^2  \\
& +  \lambda^{\frac{2}{\gamma+1}}\sum_{\lambda^2-|k|^2 > 0} f(\lambda, \lambda^2 - |k|^2) \| (\chi_0 + \chi_1) \hat{u}_k \|_{L^2(\Mp)}^2 ,
\end{align*}
where $f(\lambda, \omega)$ is defined in~\eqref{eq: f(lambda,omega)}. According to Lemma~\ref{lem: f positive}, there exists $\delta >0$ such that for all $\lambda >0$ and $\omega \in [0, \delta \lambda^2]$, we have $f(\lambda, \omega) \geq 1$. Moreover, we always have $\lambda^{\frac{2}{\gamma+1}}f(\lambda, \omega) \geq - c_0 \omega \geq -c_0 \lambda^2$. As a consequence, we have
\begin{align}
\label{eq: estim lambda2reste}
\|P_\lambda u\|_{L^2(M)}^2 + \|u\|_{L^2(M)}^2
&  \gtrsim 
\lambda^{\frac{2}{\gamma+1}}\sum_{\lambda^2-|k|^2 \leq \delta \lambda^2} \| \hat{u}_k\|_{L^2(\Mp)}^2 
 -  c_0 \lambda^2 \sum_{\delta \lambda^2 <\lambda^2 - |k|^2 \leq \lambda^2} \| \hat{u}_k \|_{L^2(\Mp)}^2 ,
\end{align}
Next, we study the last term in this estimate. The range $\delta \lambda^2 <\lambda^2 - |k|^2 \leq \lambda^2$ may be rewritten as $\left|\frac{k}{\lambda}\right|^2 \leq (1-\delta)$. Take $\chi \in \Cinfc(\R)$ such that $\chi = 1$ in a \nhd of $[-(1-\delta) , (1-\delta) ]$ and $\chi = 0$ in a \nhd of $(-\infty , -1]\cup[1 , + \infty)$. We then have
\begin{align}
\label{eq: mult fourier}
\sum_{\delta \lambda^2 <\lambda^2 - |k|^2 \leq \lambda^2} \| \hat{u}_k \|_{L^2(\Mp)}^2 \leq \sum_{k \in \Z } \chi^2 (|k|^2/ \lambda^2) \| \hat{u}_k \|_{L^2(\Mp)}^2 =  \| \chi(-\Delta_{\T^{n''}}/\lambda^2) u \|_{L^2(M)}^2 , 
\end{align}
where the last identity is a consequence of Lemma~\ref{lem: pseudo circle}. 
Using~\eqref{e:fourmult}, we moreover have 
\begin{equation}
\label{e:multquantiz}
\| \chi(-\Delta_{\T^{n''}}/\lambda^2) u \|_{L^2(M)}^2 = \| \Op \left(\chi(|\xi''|^2/\lambda^2) \right) u \|_{L^2(M)}^2 + O(\lambda^{-2}) \|u \|_{L^2(M)}^2 .
\end{equation}
As $\supp(\chi) \subset (-1 , 1)$, Assumption~\eqref{eq: GCC outside} implies that for all $\rho \in \supp(1 \otimes \chi ) \cap S^*M$ there exists $t > 0$ such that $\phi_{t}(\rho) \in T^*\omega_b$. We may hence apply Lemma~\ref{lem: geom control} with $\alpha = 1 \otimes \chi$ to obtain
\begin{align}
\label{eq: mult + propag}
\lambda^2 \| \Op \left( \chi(|\xi''|^2/\lambda^2)\right) u \|_{L^2(M)}^2 \leq C \|P_\lambda \Op \left(\chi(|\xi''|^2/\lambda^2) \right) u \|_{L^2(M)}^2 + O(\lambda^{-1}) \|u \|_{L^2(M)}^2  .
\end{align}
We now have 
\begin{align}
\label{eq: P chi chi P}
P_\lambda \Op \left(\chi(|\xi''|^2/\lambda^2)\right)& = (-\Delta_{\Mp} -\Delta_{\T^{n''}} - \lambda^2 + i \lambda b) \Op \left(\chi(|\xi''|^2/\lambda^2)\right)\nonumber \\
& =  \Op \left(\chi(|\xi''|^2/\lambda^2)\right) P_\lambda+  [-\Delta_{\T^{n''}},  \Op \left(\chi(|\xi''|^2/\lambda^2)\right)]  + i \lambda [b, \Op \left(\chi(|\xi''|^2/\lambda^2)\right)] ,
\end{align}
as $[\Delta_{\Mp}, \Op \left(\chi(|\xi''|^2/\lambda^2)\right) ] = 0$. Next, we have  $[-\Delta_{\T^{n''}}, \Op \left(\chi(|\xi''|^2/\lambda^2)\right) ] \in \Psi^0_{sc}(M)$ as its principal symbol is $\{|\xi''|^2, \chi(|\xi''|^2/\lambda^2)\}=0$.
Finally, we have 
\begin{align}
\label{eq: estim comm b}
\left\| i \lambda [b,  \Op \left( \chi(|\xi''|^2/\lambda^2) \right)] \right\|_{\calL(L^2(M))} \leq C\big( \|b\|_{L^\infty(M)} + \|\nabla_{x''} b\|_{L^\infty(M)} \big) ,
\end{align}
uniformly with respect to $\lambda$ according to~\cite[Chapter~IV, Proposition~7]{CM:78}.
Combining all estimates~\eqref{eq: mult fourier}-\eqref{eq: estim comm b}, we now obtain 
\begin{align*}
\lambda^2 \sum_{\delta \lambda^2 <\lambda^2 - |k|^2 \leq \lambda^2} \| \hat{u}_k \|_{L^2(\Mp)}^2 \lesssim \|P_\lambda u \|_{L^2(M)}^2 + \|u \|_{L^2(M)}^2 .
\end{align*}
Together with~\eqref{eq: estim lambda2reste}, this  now implies 
\begin{align*}
\|P_\lambda u\|_{L^2(M)}^2 + \|u\|_{L^2(M)}^2
&  \gtrsim \lambda^{\frac{2}{\gamma+1}}\sum_{\lambda^2-|k|^2 \leq \delta \lambda^2} \| \hat{u}_k\|_{L^2(\Mp)}^2 
 +  \lambda^2 \sum_{\delta \lambda^2 <\lambda^2 - |k|^2 \leq \lambda^2} \| \hat{u}_k \|_{L^2(\Mp)}^2 \\
 &  \gtrsim \lambda^{\frac{2}{\gamma+1}} \| u\|_{L^2(M)}^2 ,
\end{align*}
which concludes the proof of the theorem when taking $\lambda \geq \lambda_0$ for $\lambda_0$  large enough.
\end{proof}
}
\section{Proof of Theorem~\ref{th: sequence saturation}: the lower bound}
\label{sec:quasimodes}
{
In this section, we prove Theorem~\ref{th: sequence saturation}, \ie we construct a sequence of quasimode that saturates Inequality~\eqref{eq:expected estimate}.
Let $\chi \in \Cinfc(B(0,\eps_0))$ be a real-valued function such that $\chi = 1$ in a \nhd of $0$.
Take $v = (1, 0, \cdots,0) \in \Z^{n''}$, $\lambda =k$ for $k \in \N$ and set
$$
u_k(x', x'') = c_1 k^{\frac{\alpha (n-1)}{2}} \chi(k^\alpha x') \frac{e^{ik v \cdot x''}}{(2\pi)^{\frac{n''}{2}}} = c_1 k^{\frac{\alpha (n-1)}{2}} \chi(k^\alpha x') \frac{e^{ik x_1''}}{(2\pi)^{\frac{n''}{2}}},
$$
where $c_0 = \|\chi\|_{L^2(B(0,\eps_0))}^{-1}$ is chosen so that for any $k \in \N$, we have $\|u_k\|_{L^2(M)} =1$, and $\alpha>0$ is to be fixed later on (in terms of $\gamma$).
Now, we have 
$$
(P_k u_k)(x) =
c_0 k^{\frac{\alpha n'}{2}} \frac{e^{ik x_1''}}{(2\pi)^{\frac{n''}{2}}} \left(- k^{2\alpha} (\Delta_{\Mp} \chi)(k^\alpha x') 
+ i k b(x')\chi(k^\alpha x')\right),
$$
}
and
\begin{align*}
\|P_k u_k\|_{L^2(M)}^2 & = c_0^2 \int_{\Mp} k^{\alpha n'}\left|- k^{2\alpha} (\Delta_{\Mp} \chi)(k^\alpha x') 
+ i k b(x')\chi(k^\alpha x')\right|^2 dx' \\
& = c_0^2 \int_{\R^{n'}}  k^{\alpha n'}\left| k^{2\alpha} (\Delta_{\Mp} \chi)(k^\alpha x') \right|^2 dx'
+ c_0^2 \int_{\R^{n'}}  k^{\alpha n'}\left| k b(x')\chi(k^\alpha x') \right|^2 dx'
\end{align*}
since $\chi$ is real-valued and compactly supported. After a change of variables, this yields
\begin{align*}
\|P_k u_k\|_{L^2(M)}^2 = c_0^2 \int_{\R^{n'}} \left| k^{2\alpha} (\Delta_{\Mp} \chi)(x') \right|^2 dx'
+ c_0^2 \int_{\R^{n'}}  \left| k b\big(\frac{x'}{k^\alpha} \big)\chi( x') \right|^2 dx' .
\end{align*}
{Using then Assumption~\eqref{eq: local homo gamma} on the vanishing rate of $b$ on $\supp(\chi)$, we obtain
\begin{align*}
\|P_k u_k\|_{L^2(M)}^2 & \leq c_0^2 k^{4\alpha} \int_{\R^{n'}} \left|  (\Delta_{\Mp} \chi)(x') \right|^2 dx'
+ c_0^2  \int_{\R^{n'}}  \left| k C_1\left| \frac{x'}{k^\alpha} \right|^{2\gamma}\chi( x') \right|^2 dx' \\
 & \leq c_0^2 k^{4\alpha} \int_{\R^{n'}} \left|  (\Delta_{\Mp} \chi)(x') \right|^2 dx'
+ c_0^2C_1^2  k^{2-4\alpha\gamma} \int_{\R^{n'}}  \left|  |  x' |^{2\gamma}\chi( x') \right|^2 dx' .
\end{align*}
Minimizing the exponent \wrt $\alpha$ gives $\alpha = \frac{1}{2(1+\gamma)}$, and hence
\begin{align*}
\|P_k u_k\|_{L^2(M)}^2 \leq  k^{\frac{2}{1+\gamma}}  c_0^2 \left(  \| (\Delta_{\Mp} \chi)\|_{L^2(\Mp)}^2 + C_1^2 \| |  x' |^{2\gamma}\chi\|_{L^2(\Mp)}^2 \right).
\end{align*}
Finally, we have $\|P_k u_k\|_{L^2(M)} \leq  C_0 k^{\frac{1}{1+\gamma}}$ for some $C_0>0$, which concludes the proof of Theorem~\ref{th: sequence saturation}.
}
\section{Second microlocalization, a key tool for the proof}
\label{sec:2ndmicro}
The proofs above  contain several steps of microlocalizations, i.e. of cutting the phase space 
into pieces, proving the key estimates for symbols supported in these pieces and finally patching together the whole set of inequalities. 
We are willing in this section to (hopefully) illuminate these technicalities by resorting to various concepts related to the so-called second microlocalization.
These notions were first developed in the analytic category 
in  M.~Kashiwara \& T.~Kawai's article \cite{MR579740}, followed by G.~Lebeau's paper \cite{MR786539}.
J.-M.~Bony's article \cite{MR925240} and J.-M. Delort's book \cite{MR1186645}
displayed striking applications to propagation of  weak singularities for non-linear equations,
the J.-M.~Bony \& N.~Lerner's paper   \cite{MR1011988} provided  a metrics point of view.
More recently, 
N.~Anantharaman \& M.~L{\'e}autaud's work on the damped wave equation \cite{AL:13}
showed, using techniques of N.~Anantharaman \& F.~Maci\`a \cite{Macia:10,AM:11}, that the second micolocalization 
could be useful to tackle estimates
related to some non-selfadjoint operators.
{
The key tools in the last three papers are the 2-microlocal measures, introduced by L.~Miller~\cite{Miller:97}, C.~Fermanian-Kammerer and P.~G{\'e}rard~\cite{Ferm:00,FKG:02,Ferm:05}, which allow to perform (at the level of defect measures) a second microlocalization for bounded sequences in $L^2$.

In the present article, we have used a multiplier method, instead of resorting to $2$-microlocal measures: it means that have computed
$$
\poscal{P_{\lambda}u}{\mathcal Mu},
$$
with a carefully chosen {\it multiplier} $\mathcal M$.
That multiplier  operator $\mathcal M$ is in fact a second microlocalization operator 
and cannot be chosen as a standard semi-classical operator.
It is constructed rather explicitly
in the various regions of the phase space.
}
\subsection{First microlocalization arguments and their limitations}
Taking advantage of the fact 
that the 
damping term $b$ in \eqref{eq: stabilization}
does not depend on time\footnote{It would be interesting to explore the case where $b$ depends on time
and to show that a direct energy method should provide essentially the same results, at least when the time-dependence is smooth.},
a Fourier-Laplace transform yields the operator $P_{\lambda}$
\begin{equation}\label{2-jjkk}
 P_{\lambda}=-\Delta_{g}-\lambda^{2}+ib(x)\lambda,
\end{equation}
where $\lambda$ can be considered as a large positive parameter.
This is thus a semi-classical problem where the Planck constant is $h=1/\lambda$.
$P_{\lambda}$
 is a real principal type operator whose principal symbol is 
$
p=\val{\xi}_{x}^{2}-\lambda^{2}
$
so that the characteristic manifold is
\begin{equation}\label{2-6547}
\Char P_\lambda =\{(x,\xi)\in T^{*}M,\ \val{\xi}_{x}^{2}=\lambda^{2}\}.
\end{equation}
We ask the following regularity question: assuming that $P_{\lambda} u$ belongs to the (semi-classical) Sobolev space $H^{s}_{sc}$,
could we have $u\in H^{s+1}_{sc}$ at a point $\gamma_{0}$ of the cotangent bundle?
Of course when $\gamma_{0}$ is non-characteristic, we have the better result $u\in H^{s+2}_{sc}$.
If $\gamma_{0}$ belongs to the characteristic set,
we may combine three pieces of information:
\vs
[1] The point $\gamma_{0}$ belongs to a bicharacteristic curve of $p$ whose endpoint $\gamma_{1}$
belongs to the set where the imaginary part of the subprincipal symbol 
is elliptic.
\par
[2] At $\gamma_{1}$, we have a regularity result, due to the ellipticity
of
the imaginary part of the subprincipal symbol.
\par
[3] The propagation-of-singularities theorem for real-principal type operators
shows that the regularity at $\gamma_{1}$ propagates down to $\gamma_{0}$.
\vs
For condition [1] to be fulfilled,
we need the hypothesis that the point $\gamma_{0}$ is connected by the bicharacteristic flow to the set where $b$ is positive.
The so-called {\it Geometric Control Condition}
requires that this should be true for any point $\gamma_{0}$ in the characteristic set,
providing then a (semi-classical)
regularity result.
We may thus introduce the {\it singular set}
\begin{equation}\label{}
S_\lambda =\{(x,\xi)\in \Char P_\lambda, \ \forall  t\in \R,\ b(\phi_{t}(x,\xi))=0\},
\end{equation}
(here $\phi_{t}$ is the bicharacteristic flow)
and try to understand what will happen if $S_{\lambda}$ is not empty.
{
Let us describe the model situation in which we are interested: the manifold and the operator are simply $M=\R^n = \R^{n'} \times \R^{n''}$ and  
\begin{equation}\label{2-65az}
P_{\lambda}=\op{\val{\xi'}^{2}+\val{\xi''}^{2}-\lambda^{2}+i\lambda \val{x'}^{2\gamma}} .
\end{equation}
The characteristic manifold has the equation 
$$
\val{\xi'}^{2}+\val{\xi''}^{2}=\lambda^{2}
$$
and on the characteristic curves
$
\dot{x'}=2\xi',\ \dot x''=2\xi'',\ \xi'=\text{constant}, \ \xi''=\text{constant}.
$
We get that $(x',x'',\xi', \xi'')\in S_\lambda$ iff
$
\xi'=0, \val{\xi''}^{2}=\lambda^{2}, x'=0.
$
We see that $S_{\lambda}$ is a $(2n'' - 1)$-dimensional submanifold of the symplectic
$\R^{n'}_{x'}\times\R_{x''}^{n''}\times\R^{n'}_{\xi'}\times\R_{\xi''}^{n''}$
given by
$$
S_{\lambda}=\{0_{\R^{n'}}\}\times\R_{x''}^{n''}\times\{0_{\R^{n'}}\}\times\{\xi'' \in \R^{n''} , \val{\xi''}=\lambda\} .
$$
To be 1-microlocally away from $S_{\lambda}$ would mean that for some $\epsilon_{0}\in (0,1)$
$$
\val{\xi'}^{2}+\val{\xi''}^{2}=\lambda^{2},\quad \val{\xi''}^{2}\le (1-\epsilon_{0})\lambda^{2}, \ \val{\xi'}^{2}\ge \epsilon_{0}\lambda^{2},
$$
so that the (GCC)
 condition holds even if $x'=0$. Of course if $\val{x'}>\epsilon_{0}$, the (GCC) condition holds. 
 \vs
 We are left 
 with a \nhd of  the set $S_{\lambda}$
 and we shall not be able to  take advantage of the particular behavior of $b$
 if we do not make  further localization in the phase space.
 \subsection{Reviewing our estimates}\label{sec.77}
 Let us quickly review  our arguments for the various estimates proven for $P_{\lambda}$ given by 
 \eqref{2-65az}.
 We set $\val{\xi'}^{2}+\val{\xi''}^{2}=\lambda^{2}$.
 \vs\no {\bf [1] 
$\val{\xi'}^{2}\gtrsim \val{\xi''}^{2}$ or $\val{x'}\gtrsim 1$}: this is the (GCC) region
since the first condition implies that the characteristic curve starting at $x'=0$ enters at once the damping set, and the second condition requires to start within the damping set.
\vs\no {\bf [2] 
$\val{\xi'}^{2}\ll \val{\xi''}^{2}$ and $\val{x'}\ll1$}: this is where we need further localization.
\vs
 \begin{itemize}
 \item[]{\bf  [2.1]}
$\val{\xi'}^{2}\ll \lambda^{\frac1{\gamma+1}}$ and $\val{x'}^{2}\ll\lambda^{-\frac1{\gamma+1}}$: a tiny piece of the phase space, with volume 1 though, thus compatible with the uncertainty principle. We use a pseudo-spectral estimate for the operator
$\val{D'}^{2}+i\lambda\val{x'}^{2\gamma}$, that is an estimate of type
$$
\norm{\val{D'}^{2}u+i\lambda\val{x'}^{2\gamma}u}\gtrsim \norm{u}\lambda^{\frac1{\gamma+1}}.
$$
\vs
\item[]
{\bf  [2.2]} $\val{\xi'}^{2}\ll \val{\xi''}^{2}$ and 
$\lambda^{-\frac1{\gamma+1}}\lesssim \val{x'}^{2}\ll1$: there we have
 $\lambda\val{x'}^{2\gamma}\gtrsim \lambda^{\frac1{\gamma+1}}$, we calculate 
 $$\re\poscal{P_{\lambda}u}{iu}=\poscal{\lambda\val{x'}^{2\gamma}u}{u}\gtrsim
 \lambda^{\frac1{\gamma+1}}\norm{u}^{2}.$$
\vs
\item[]{\bf  [2.3]}
$\lambda^{\frac1{\gamma+1}}\lesssim \val{\xi'}^{2}\ll \val{\xi''}^{2} $ and $\val{x'}^{2}\ll\lambda^{-\frac1{\gamma+1}}$: this is the most difficult region, in which we use a propagation estimate.
We study the model
$$
\val{D'}^{2}-\omega^{2}+i\lambda \val{x'}^{2\gamma},\quad \lambda^{\frac1{2\gamma+2}}\le \omega\le \epsilon_{0}\lambda.
$$
\end{itemize}
\begin{figure}[h!]\label{figure}
  \begin{center}
    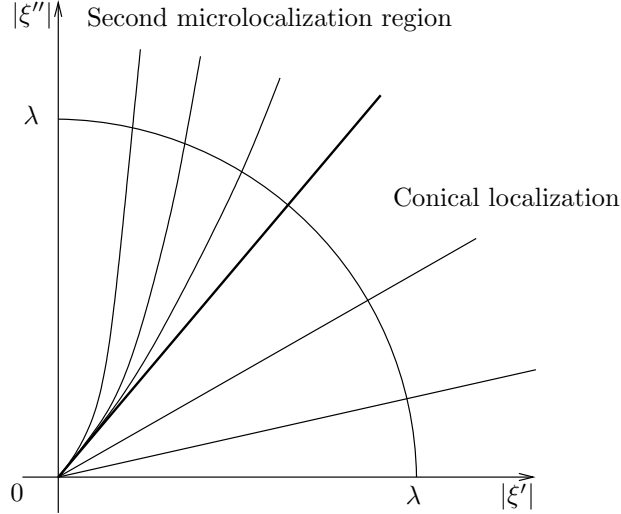 
    \caption{{Classical conical localization and second microlocalization on the singular set.}}
 \end{center}
\end{figure}
The above picture illustrates the transition from region {\bf [1]} to region {\bf [2}] in the frequency variables for small $x'$
and  shows that localization in region {\bf [2]} is no longer conical.
\subsection{Second microlocalization with respect to the singular set}
$\ast$ {The first microlocalization metric on $\R^{n}_{x}\times\R^{n}_{\xi}$
 is}
\begin{equation}\label{}
G_{x,\xi}=\frac{\val{dx}^{2}}{1}+\frac{\val{d\xi}^{2}}{\Lambda(\xi)^{2}},
\text{
with
$\Lambda(\xi)^{2}=\valjp \xi^{2}=1+\val \xi^{2}$.}
\end{equation}
It means that the standard symbols of order $0$
used for this first microlocalization 
are functions $a\in \moo(\RZ)$ such that
$$
\val{(\p_{x}^{\alpha}\p_{\xi}^{\beta} a)(x,\xi)}\le C_{\alpha\beta}\Lambda(\xi)^{-\val \beta}=C_{\alpha\beta}
\valjp{\xi}^{-\val \beta}.
$$
The ``large parameter'' of this calculus is the product of the conjugate axes, {$\Lambda=\Lambda(\xi)$.}
\vs\no
$\ast$
We want to provide a finer localization when we are getting close to  the singular set $S_\lambda$.}
For this purpose, we define the metric
\begin{align}
&g_{x,\xi}=\frac{\val{dx'}^{2}}{\Lambda(\xi)^{-\frac{1}{\gamma+1}}\mu(\xi)^{\frac{1}{\gamma+1}}}
+\frac{\val{d\xi'}^{2}}{\Lambda(\xi)^{\frac{1}{\gamma+1}}\mu(\xi)^{\frac{2\gamma+1}{\gamma+1}}}
+\frac{\val{dx''}^{2}}{1}+\frac{\val{d\xi''}^{2}}{\Lambda(\xi)^{2}},\label{eee}\\{~}\notag
\\
&\text{where}\quad
\mu(\xi)=1+\bigl(\val{\xi'}^2
\Lambda(\xi)^{-\frac{1}{\gamma+1}}\bigr)^{{\frac{\gamma+1}{2\gamma+1}}}.
\end{align}
The notation for $g$ above means that for each $(x,\xi)\in \RZ$,
$g_{x,\xi}$ is a positive definite quadratic form on $\RZ$
so that 
for $(z',z'',\zeta',\zeta'')\in \R^{n'}\times\R^{n''}\times\R^{n'}\times \R^{n''}$, we have
$$
g_{x,\xi}(z,\zeta)=
\frac{\val{z'}^{2}}{\Lambda(\xi)^{-\frac{1}{\gamma+1}}\mu(\xi)^{\frac{1}{\gamma+1}}}
+\frac{\val{\zeta'}^{2}}{\Lambda(\xi)^{\frac{1}{\gamma+1}}\mu(\xi)^{\frac{2\gamma+1}{\gamma+1}}}
+\frac{\val{z''}^{2}}{1}+\frac{\val{\zeta''}^{2}}{\Lambda(\xi)^{2}}.
$$
We note first that
\begin{equation}\label{222}
1\le \mu(\xi)\le 1+\Lambda(\xi)\le 2\Lambda(\xi).
\end{equation}
This inequality implies 
that 
$$
\Lambda^{-1}\mu\le 2,\quad 
\Lambda^{\frac{1}{\gamma+1}}\mu^{\frac{2\gamma+1}{\gamma+1}}\le \Lambda^{2 }\ 2^{\frac{2\gamma+1}{\gamma+1}},
\text{ and hence }
g\ge c(\gamma) G,
$$
where $c(\gamma)$
is a positive constant depending only
on  $\gamma$.
This inequality induces of course that the localization given by the metric $g$ is finer than the one provided by $G$.
Calculating the square of the product of conjugate axes of $g$, we get
respectively
$$
\Bigl(\Lambda(\xi)^{-\frac{1}{\gamma+1}}\mu(\xi)^{\frac{1}{\gamma+1}}
\Bigr)\times 
\Bigl(\Lambda(\xi)^{\frac{1}{\gamma+1}}\mu(\xi)^{\frac{2\gamma+1}{\gamma+1}}\Bigr)= 
\mu(\xi)^2,
\quad1\times \Lambda(\xi)^{2}
$$
so that the ``large parameter'' of the metric $g$ is
(equivalent to) $\mu$, and
$g$ satisfies the Uncertainty Principle.
\begin{lemma}\label{77}
The metric $g$ given by \eqref{eee}
is slowly varying, i.e. such that there exist $r, C$
positive  so that 
\begin{equation}\label{1111}
g_{y,\eta}\bigl((y,\eta)-(x,\xi)\bigr)\le r^{2}\Longrightarrow \forall T\in \RZ, \quad C^{-1}\le \frac{g_{x,\xi}(T)}{g_{y,\eta}(T)}\le C.
\end{equation}
Moreover, the metric $g$ is also uniformly temperate on the balls of the metric $G$, i.e.
there exists $C,N, r$ positive
such that
$\forall X=(x,\xi)\in \RZ, \forall Y=(y,\eta)\in \RZ,\quad$
\begin{equation}\label{2222}
G_{X}(Y-X)\le r^{2}
\Longrightarrow \forall T\in \RZ, \quad 
\frac{g_{x,\xi}(T)}{g_{y,\eta}(T)}\le C\bigl(1+g_{x,\xi}^{\sigma}(Y-X)\bigr)^{N},
\end{equation}
 where the quadratic form $g_{x,\xi}^{\sigma}$ (which is the ``symplectic inverse'' of $g_{x,\xi}$)
is given by
$$
g_{x,\xi}^{\sigma}(z,\zeta)=
\Lambda(\xi)^{\frac{1}{\gamma+1}}\mu(\xi)^{\frac{2\gamma+1}{\gamma+1}}
{\val{z'}^{2}}
+\Lambda(\xi)^{-\frac{1}{\gamma+1}}\mu(\xi)^{\frac{1}{\gamma+1}}{\val{\zeta'}^{2}}
+\Lambda(\xi)^{2}{\val{z''}^{2}}+{\val{\zeta''}^{2}}.
$$ 
\end{lemma}
\begin{proof}
 We have
 \begin{equation}\label{3333}
\left(\frac{g_{x,\xi}}{g_{y,\eta}}\right)^{\gamma+1}\lesssim\frac{\Lambda(\eta)}{\Lambda(\xi)}\frac{\mu(\eta)^{2\gamma+1}}{\mu(\xi)^{2\gamma+1}}
 +\frac{\Lambda(\xi)}{\Lambda(\eta)}\frac{\mu(\eta)}{\mu(\xi)},
\end{equation}
 and if 
 $g_{y,\eta}\bigl((y,\eta)-(x,\xi)\bigr)\le r^{2}$,
 this implies 
 $G_{y,\eta}\bigl((y,\eta)-(x,\xi)\bigr)\le r^{2}/c(\gamma)$.
 Since $G$ is slowly varying, we may choose $r$ small enough to 
 get
 $
{\Lambda(\xi)}/{\Lambda(\eta)}\sim 1.
 $
 Then, with fixed constants $C_{j}$, we find 
 $$
 \frac{\mu(\eta)^{2\gamma+1}}{\mu(\xi)^{2\gamma+1}}\le C_{1} \frac{1+\val{\eta'}^{2\gamma+2}\Lambda^{-1}}{
 1+\val{\xi'}^{2\gamma+2}\Lambda^{-1}}\le 
C_{2}+
C_{2} \left(\frac{\val{\eta'-\xi'}^{2}\Lambda^{-\frac1{\gamma+1}}}{1+\val{\xi'}^{2}\Lambda^{-\frac1{\gamma+1}}}\right)^{\gamma+1},
 $$
 and since from the assumption \eqref{1111}, we know that
 $
 \val{\xi'-\eta'}^{2}\le r^{2}\mu(\eta)^{\frac{2\gamma+1}{\gamma+1}}\Lambda^{\frac1{\gamma+1}},
 $
 we obtain
 $$
  \frac{\mu(\eta)^{2\gamma+1}}{\mu(\xi)^{2\gamma+1}}\le C_{2}+C_{3}r^{2\gamma+2}\frac{\mu(\eta)^{2\gamma+1}}{\mu(\xi)^{2\gamma+1}}
  \quad\text{which implies }
  \frac{\mu(\eta)^{2\gamma+1}}{\mu(\xi)^{2\gamma+1}}\le 2C_{2},
 $$
 if
 $C_{3}r^{2\gamma+2}\le 1/2$, providing the inequality $g_{X}\le C_{4} g_{Y}$ in \eqref{1111} for $r\le r_{0}$.
 Now if $g_{Y}(X-Y)\le r_{0}^{2}/C_{4}$, we get 
 $g_{X}(X-Y)\le r_{0}^{2}$ and thus $g_{Y}\le C_{4}g_{X}$, completing the proof of \eqref{1111}.
 \par
 To prove \eqref{2222},
 we may choose $r$ such that $G_{X}(Y-X)\le r^{2}$ implies $\Lambda(\xi)\sim \Lambda(\eta)$
 and from \eqref{3333}, it suffices  to prove
 $$\val{\xi'-\eta'}^{2}\Lambda^{-\frac1{\gamma+1}}\le
\val{\xi'-\eta'}^{2}\Lambda^{-\frac1{\gamma+1}}\mu(\xi)^{\frac1{\gamma+1}},
 $$
 which is true since $\mu\ge 1$.
\end{proof}
\begin{remark}
Lemma \ref{77} may look outrageously complicated and unintuitive, but these properties,
essentially introduced by L. H\"ormander (see Chapter 18, in \cite{Hoermander:V3}),
are linked to some ``admissibility'' of the cutting of the phase space provided by this metric.
The uncertainty principle (here $\mu\ge 1$) is the most natural condition,
but Conditions \eqref{1111} and \eqref{2222}
are important for a pseudodifferential calculus to make sense. In other words,
we need some conditions to patch together the estimates that we are able to prove in each specific
region described in Section \ref{sec.77}.
A cutting procedure will generate commutators and we have to make sure that these commutators do not destroy or spoil the basic local estimates that we are able to prove.
\end{remark}
 \appendix
\section{Some geometric facts}
\label{s:geometry}
In this section, we recall some elementary geometric definitions and facts.
The metric $g$ furnishes in each tangent space $T_xM$ (as well as in each cotangent space $T^*_xM$) an inner product denoted $(\cdot , \cdot)_{g(x)}$ (with the same notation on $T^*_xM$).
In local coordinates, we write $g_{ij}$ for the metric $g$ on the
tangent bundle $TM$. As a metric on the cotangent bundle $T^*M$, $g$
is given by $g^{ij}$ in local coordinates, \ie 
$$
(\eta,\xi)_{g(x)} = \ssum_{i,j} g^{ij}(x) \eta_i \xi_j, \quad \text{ where } g^{ij}(x) = (g(x)^{-1})_{ij} . 
$$
For $x \in M$ and $\xi \in T^*_x M$, we denote by $|\xi|_x = (\xi,\xi)_{g(x)}^{\frac12}$ the associated norm. We also use the notation $p(x,\xi) = |\xi|_x^2$.
For all $v \in T_xM$, we can define $v^* \in T^*_xM$ uniquely determined
by the identity
$$ (v, w)_{g(x)} = \left< v^* , w \right>_{T_{x}^{*}(M), T_{x}(M)}, \qquad\text{for all $w \in T_xM$,}$$
which reads in local coordinates $v^*_i = \sum_j g_{ij}(x) v_j$. Note that $|v|_x = |v^*|_x$.
\vs
We now give a definition of geodesics on $M$ associated with the metric $g$, as used in Theorem~\ref{th: resolvent estimate}. 
We denote by $s \mapsto \phi_s(x , \xi) \in T^*M \setminus 0$ the Hamiltonian flow associated to $p$, that is, the (maximal) solutions of
\begin{equation}
  \label{eq: geodesic flow}
  \frac{d}{ds}\phi_s(x, \xi) = H_p \big(\phi_s(x, \xi)\big), \quad \phi_0(x, \xi) = (x, \xi)  \in T^*M \setminus 0, 
\end{equation}
where the Hamilton vector field $H_p$ is given by $H_p = (\nabla_{\xi}p , - \nabla_{x}p )$ in local coordinates. In particular we have $\frac{d}{ds} x_i(s) = 2 \sum_j g^{ij}(x(s)) \xi_j(s)$, that is $\xi(s) = \frac12 \big( \frac{d}{ds} x(s)\big)^*$. Note that the value of $p$ is preserved along this integral curve as 
$$
 \frac{d}{ds}p \circ\phi_s|_{s = s_0} = H_p(p)(\phi_{s_0}) = \{p,p\}(\phi_{s_0}) =0 .
$$
As a consequence, $\phi_s$ is a global flow preserving the norm.
\par
Let now $S^2M = \{(x,v) \in TM , |v|_x =(v,v)_{g(x)}^\hf = 2\}$. For $(x,v) \in S^2M$, we consider the curve $(x(s),v(s))$ given by
$$
(x(s), v(s)^*) = \phi_s(x, v^*).
$$
Note that we have $\frac{d}{ds} x(s) = v(s)$. In particular, $\frac{d}{ds} x(0) = v$ and moreover 
$$
|v(s)|_{x(s)} = |v(s)^*|_{x(s)} = 2 |\xi(s)|_{x(s)} = p({x(s)},v(s)^*)^\frac12 =  p(x,v^*)^\frac12 = |v|_x = 2.
$$
We call the curve $s \mapsto x(s)$ on $M$ the geodesic originating from $(x, \xi)\in S^*M$ at time $s=0$. The above remarks show that $(x(t), \frac{dx}{dt}(t)) \in S^2M$: the traveling speed of the geodesic is constant (and equal to $2$).
\par
Finally, the covariant gradient and the divergence operators are given in local
coordinates by
$$
\nabla_g = \ssum_i g^{ij} \d_{x_i}, \qquad 
\div_g v = \frac{1}{\sqrt{\det(g)}} \ssum_i \d_{x_i}( \sqrt{\det(g)} v_i),
$$
The usual (negative) Laplace-Beltrami operator on $M$ is defined by $\Delta_g = \div_g \nabla_g$, that is,
$$
\Delta_g = \frac{1}{\sqrt{\det(g)}}
\ssum_{i,j} \d_{x_i}(g^{ij} \sqrt{\det(g)}\d_{x_j}),
$$ in local coordinates. {It is selfadjoint on $L^2(M)$ endowed with the Riemannian dot-product $\langle f ,g \rangle_{L^2(M)}$ given by  $\int f \overline{g} \sqrt{\det(g)} dx$ in local charts.
\par
Finally, under the additional structure assumption $(M,g) = ( \Mp \times \T^{n''}, g' + |d x_1''|^2 + \cdots + |d x_{n''}''|^2)$ (where $g'$ is a metric on $M'$), we obtain
$$
\phi_s(x', \xi' ,x'',\xi'') =(\phi_s'(x', \xi') ,x'' + s \xi'',\xi'') ,
$$
(where we changed the order of the variables for readability) with $\phi_s'$ the flow on  $T^*\Mp$ associated with the Hamilton vector field $H_{p'}$ with $p'(x',\xi') = |\xi'|_{x'}^2$.
Similarly, we have
$$
\Delta_g = \Delta_{M'} +\Delta_{\T^{n''}} = \frac{1}{\sqrt{\det(g')}}
\ssum_{i,j \leq n'} \d_{x_i'}({g'}^{ij} \sqrt{\det(g')}\d_{x_j'}) + \ssum_{j \leq n''} \d_{x_i''}^2.
$$ 
}
\section{Toolbox of pseudodifferential calculus}
\label{s:pseudodiff}
\subsection{Pseudodifferential operators on \texorpdfstring{$\R^d$}{rd}}\label{sec.pseudo}
{\bf Notations.}
We recall that the Weyl quantization of a symbol $a(x,\xi)$ on $\R^{2d}$, the  operator denoted by $a^{w}$, is given by 
\begin{align}
\label{eq: weyl quantiz}
(a^{w}u)(x)=\iint e^{i(x-y)\cdot \xi} a(\frac{x+y}2,\xi) u(y) dy d\xi (2\pi)^{-d},
\end{align}
which is a small variation with respect to the more standard quantization
$$
(a(x,D)u)(x)=\int e^{ix\cdot \xi} a(x,\xi) \hat u(\xi)  d\xi (2\pi)^{-d}.
$$
One of the (many) assets of Weyl quantization is the formula for taking adjoints, $$(a^{w})^{*}=(\bar a)^{w},$$ a convenient feature for our computations with non-selfadjoint operators.
The symplectic invariance of the Weyl quantization
is an important property, useful for the proof that our estimate \eqref{eq:estimate Rd}
is optimal, can be expressed as follows (see e.g. \cite[Theorem 2.1.2]{Lernerbook}):
let $a$ be a tempered distribution on $\R^{2d}$ and let $\chi$ be an affine symplectic mapping of $\R^{2d}$. Then there exists a unitary transformation $U$ of $L^{2}(\R^{d})$
such that
\begin{equation}\label{segal}
(a\circ \chi)^{w}=U^{*}a^{w}U.
\end{equation}
This implies for instance that, for $\alpha_{0}>0, (x_{0},\xi_{0})\in \R^{2d}$,
the operator with Weyl symbol $b$ given by
$$
b(x,\xi)=a(\alpha_{0} x+x_{0},\alpha_{0}^{-1}\xi+\xi_{0})
$$
is unitarily equivalent to $a^{w}$.
In the main part of the article, we also use the notation 
\begin{equation}\label{bfps}
S(\mathcal{M},\frac{\val {dx}^{2}}{\varphi(x,\xi)^{2}}+\frac{\val {d\xi}^{2}}{\Phi(x,\xi)^{2}})
\end{equation}
for the space of smooth functions $a$
on $\R^{2d}$ such that for each multi-indices $\alpha,\beta$, there exists $C_{\alpha\beta}>0$ such that
$$\forall (x,\xi)\in \R^{2d},\quad
\val{(\p_{x}^{\alpha}\p_{\xi}^{\beta} a)(x,\xi)}\le C_{\alpha\beta}\mathcal{M}(x,\xi)\varphi(x,\xi)^{-\val \alpha}\Phi(x,\xi)^{-\val \beta},
$$
where the positive functions $\varphi,\Phi, \mathcal{M}$ are such that the metric 
$\frac{\val {dx}^{2}}{\varphi(x,\xi)^{2}}+\frac{\val {d\xi}^{2}}{\Phi(x,\xi)^{2}}$ and the weight $\mathcal M$ are admissible (see~\cite[Section~18.5]{Hoermander:V3} or \cite[Section~2.2]{Lernerbook} for precise definitions).
In Section~\ref{s:geometriccontrol} (where we only use ``one-microlocal'' calculus), the following semiclassical class is used (see below for its definition on a manifold):
$$
S^m_{sc}(\R^{2n}) : = S((\lambda^2 + \langle \xi \rangle^2)^\frac{m}{2},\val{dx}^{2}+\frac{\val {d\xi}^{2}}{\lambda^2 + \langle \xi \rangle^2}), \quad \lambda \geq 1 .
$$
Note that $|\xi|^2 - \lambda^2 \in S^2_{sc}(\R^{2n})$, and for any $a \in \Cinf (\R^{2n})$ which is homogeneous of degree $m$, we have $\lambda^{\tilde{m}}a(x, \frac{\xi}{\lambda}) \in S^{m+\tilde{m}}_{sc}(\R^{2n})$. We also denote by\footnote{\label{s000+}$S^0_{0,0}(\R^{2n}) $ is the space of smooth functions on $\RZ$ which are bounded as well as all their derivatives.}
\begin{equation}\label{s000}
S^0_{0,0}(\R^{2n}) : = S( 1 ,\val{dx}^{2}+\val {d\xi}^{2}) ,
\end{equation}
and 
$$
S^m_{\lambda}(\R^{2n}) : = S(\lambda^m,\val{dx}^{2}+\frac{\val {d\xi}^{2}}{\lambda^2 }), \quad \lambda \geq 1 .
$$
We notice that if $a(x, \xi) \in S^0_{0,0}(\R^{2n})$, then $\lambda^{m}a(x, \frac{\xi}{\lambda}) \in S^{m}_{\lambda}(\R^{2n})$. Moreover, we have $S^m_{sc}(\R^{2n}) \subset S^m_{\lambda}(\R^{2n})$
\par
We shall use also the following identities, for $a,b$ real valued symbols, say smooth functions on $\R^{2d}$
bounded with all derivatives bounded:
$$
2\re\poscal{a^{w}u}{i b^{w}u}=\poscal{[a^{w}, ib^{w}] u}{u},
$$
which follows from
$$
2\re\poscal{a^{w}u}{i b^{w}u}=\poscal{a^{w}u}{i b^{w}u}+\poscal{ib^{w}u}{a^{w}u}
=\poscal{C u}{u},
$$
with 
$
C=-i(b^{w})^{*}a^{w}+(a^{w})^{*}ib^{w}=[a^{w}, ib^{w}].
$
Moreover the ``principal''        symbol of $[a^{w}, ib^{w}]$ is the Poisson bracket 
\begin{equation}\label{poisson}
\poi{a}{b}=\sum_{1\le j\le d}\frac{\p a}{\p \xi_{j}}\frac{\p b}{\p x_{j}}-\frac{\p a}{\p x_{j}}\frac{\p b}{\p \xi_{j}}.
\end{equation}
\subsection{Pseudodifferential operators on a manifold}
\label{s:pseudoM}
In this section, we briefly explain the semiclassical calculus used in Section~\ref{s:geometriccontrol} on any $n$-dimensional compact manifold $M$. 
More details can be found in~\cite[pp~81-87]{Hoermander:V3} for classical operators or \cite[Appendix~B and Appendices C9-C13]{LRLR:11} in the semiclassical setting.
\par
Let us first recall that, given a diffeomorphism $\phi$ between two open sets $\phi: U_1 \to U_2$,
the associated pullback (here stated for continuous functions) is
\begin{align*}
  \phi^\ast : \Con (U_2) &\to \Con (U_1),\\
    u &\mapsto  u \circ \phi.
\end{align*}
For a function $a$ defined on phase-space, \eg a symbol, the pullback is
given by
\begin{align}
  \label{eq: pullback phasespace}
  \phi^\ast a (x,\xi) = a (\phi(x),\transp\phi'(x)^{-1} \xi), 
  \quad x \in U_1, \xi \in T^\ast_x(U_1) , \quad a \in \Con(T^\ast U_2).
\end{align}
Note that this transformation is symplectic.
\vs
The compact manifold $M$ is of dimension $n$ and is furnished with a finite
atlas $(U_j, \phi_j)$, $j \in J$. The maps $\phi_j: U_j \to \Ut_j
\subset \R^{n}$ are smooth diffeomorphisms. 
\begin{definition}
\begin{itemize}
We say that 
\item $a \in S^m_{sc}(T^*M)$ if $a \in \Cinf(T^*M)$ and for any $j \in J$, for any $\chi \in \Cinfc(U_j)$, we have $(\phi_j^{-1})^*(\chi a) \in S^m_{sc}(\R^{2n})$.
\item $a \in S^m_{\lambda}(T^*M)$ if $a \in \Cinf(T^*M)$ and for any $j \in J$, for any $\chi \in \Cinfc(U_j)$, we have $(\phi_j^{-1})^*(\chi a) \in S^m_{\lambda}(\R^{2n})$.
\item $a \in S^0_{0,0}(T^*M)$ if $a \in \Cinf(T^*M)$ and for any $j \in J$, for any $\chi \in \Cinfc(U_j)$, we have $(\phi_j^{-1})^*(\chi a) \in S^0_{0,0}(\R^{2n})$.
\end{itemize}
\end{definition}
Note that with this definition, we have $|\xi|_x^2 - \lambda^2 \in S^2_{sc}(T^*M)$, and for any $a \in \Cinf (T^*M)$ which is homogeneous of degree $m$ in the fibers, we have $\lambda^{\tilde{m}}a(x, \frac{\xi}{\lambda}) \in S^{m+\tilde{m}}_{sc}(T^*M)$.
\vs
Next, we explain how to quantize such symbols and recall some of the properties of the quantization.
Let us first denote by $(\psi_j)_j$ a partition of unity subordinated to the covering $M = \bigcup_{j \in J} U_j$:
 $$
 \psi_j \in \Cinf(M),\quad  \supp(\psi_j) \subset U_j,  \quad 0\leq \psi_j \leq 1, \quad \ssum_j \psi_j =1.
 $$
We also need functions $\tilde{\psi}_j \in \Cinfc(\R^{2n})$ such that $\supp(\tilde{\psi}_j) \subset \tilde{U}_j$ and  $\tilde{\psi}_j = 1$ in a \nhd of $\supp((\phi_j^{-1})^*\psi_j)$.
 \begin{definition}
Given $a = a(\lambda, x , \xi) \in \Cinf(T^*M)$, we define the following operator
$$
\Op(a) = \Op \left( a(\lambda, x , \xi) \right)  = \ssum_{j \in J} A_j , \qquad 
A_j u = \phi_j^* \big(\tilde{\psi}_j a_j^w (\phi_j^{-1})^* (\psi_j u) \big) ,\quad u \in \Cinf(M), 
$$
where $a_j = (\phi_j^{-1})^* a$.
\end{definition} 
Basically, this amounts to apply the operator associated to $a$ in local charts. This definition applies in particular if $a \in S^m_{sc}(T^*M)$ or $a \in S^m_{\lambda}(T^*M)$
Note that if $a=a(\lambda,x)$ does not depend on the cotangent variable $\xi$, then $(\Op(a) u)(x) = a(\lambda,x)u(x)$.
An operator $\Op(a)$ for $a \in S^m_{sc}(T^*M)$ is a semiclassical pseudodifferential operator in the following usual sense:
\begin{definition}
  \label{def: operator on manifold}
  We say that the operator $A: \Cinfc(M) \to \Cinf(M)$ belongs to the class $\Psi^m_{sc}(M)$ if:  \begin{enumerate}
  \item Its distribution kernel $K(x,y)$ is smooth outside $\diag (M \times M) = \{(x,x), x \in M\}$ in the semi-classical sense: $K \in \Cinf(M\times M \setminus \diag (M \times M))$ and for any semi-norm $q$ on $\Cinf(M\times M)$, for any $\chi, \hat{\chi} \in \Cinfc(M)$ such that $\supp\chi \cap \supp \hat{\chi} = \emptyset$,  we have $q\big( \chi(x) \hat{\chi} (y) K(x,y) \big) = O( \lambda^{-\infty}).$
        \item For all $j \in J$ and all $\chi \in \Cinfc( U_j)$, $\tilde{\chi} \in \Cinfc(\tilde{U}_j)$, the application
\begin{align*}
 u  \mapsto
  \big(\phi_j^{-1}\big)^\ast \big(\chi 
  A  \phi_j^\ast (\tilde{\chi} u) \big)
\end{align*}
belongs to $\Psi^m_{sc}(\R^n) = \{a^w, a \in S^m_{sc}(\R^{2n})\}$.
\end{enumerate}
\end{definition}
We could have defined as well the classes $\Psi^m_{\lambda}(M)$ to which would belong $\Op(a)$ for $a \in S^m_{\lambda}(T^*M)$. However, in the main part of the text, we shall only use two features of such operators. First, the Calder{\'o}n-Vaillancourt Theorem~\cite{CV:71} entails that for $a\in S^0_{\lambda}(T^*M)$, $\Op(a)$ defines an operator bounded on $L^2(M)$ uniformly with respect to $\lambda$ (in fact, this theorem is stated that reference on $\R^n$; its counterpart on $M$ follows easily when using local charts).
Next, we need the following result, a proof of which can be for instance adapted from~\cite[Appendix~B and Appendices C9-C13]{LRLR:11}.
\begin{proposition}
\label{p:suppdisj}
Take $a,\tilde{a} \in S^0_{\lambda}(T^*M)$ such that $d(\supp(a) , \supp(\tilde{a})) \geq C>0$ uniformly with respect to $\lambda$. Then we have
$$
\norm{\Op(a) \Op( \tilde{a})}_{\calL(L^2(M))} = O(\lambda^{-\infty}) .
$$
\end{proposition}
Proposition~\ref{p:suppdisj} implies in particular that for $a,\tilde{a} \in S^0_{0,0}(T^*M)$ such that $$\supp(a) \cap \supp(\tilde{a})=\emptyset,$$ we have
$\norm{\Op \left(a(x, \xi/\lambda) \right) \Op \left(\tilde{a}(x, \xi/\lambda) \right)}_{\calL(L^2(M))} = O(\lambda^{-\infty}) .$
\vs
Next, we describe the pseudodifferential calculus for the class $\Psi^m_{sc}(M)$. We define semiclassical norms on $M$ by
$$
\|u\|_{H^s_{sc}(M)} = \norm{\Op \big(  (|\xi|_x^2 + \lambda^2)^{\frac{s}{2}}\big) u }_{L^2(M)}.
$$
For $s\geq0$, this norm is (uniformly with respect to $\lambda \geq 1$) equivalent to the norm $\|u\|_{H^s(M)}+ \lambda^s\|u\|_{L^2(M)}$ (where $\|u\|_{H^s(M)}$ may be defined in local charts). We have the following important property for semiclassical pseudodifferential operators in the class $\Psi^m_{sc}(M)$.
\begin{proposition}
For any $s,m \in \R$ and any $A \in \Psi^m_{sc}(M)$, we have
 $$
A \in \calL(H^s_{sc}(M);H^{s-m}_{sc}(M)),\quad \text{uniformly with respect to }\lambda\geq 1.
$$
\end{proposition}
The quantization formula defining $\Op(a)$ enjoys the following properties, used in the main part of the article. We refer to~\cite{Hoermander:V3} or \cite[Appendix~B and Appendices C9-C13]{LRLR:11} for detailed proofs. 
\begin{proposition}
For any $m,\tilde{m} \in \R$ and any $a \in S^m_{sc}(T^*M)$, $\tilde{a} \in S^{\tilde{m}}_{sc}(T^*M)$, we have
\begin{itemize}
\item $\Op(a) \in \Psi^m_{sc}(M)$;
\item $\Op(a)^* - \Op(\bar{a}) \in \Psi^{m-1}_{sc}(M)$ (where the adjoint is taken in $L^2(M)$);
\item $\Op(a)  \Op(\tilde{a}) - \Op(a\tilde{a}) \in \Psi^{m +\tilde{m}-1}_{sc}(M)$;
\item $[\Op(a)  , \Op(\tilde{a}) ] - \frac{1}{i}\Op(\{a, \tilde{a}\}) \in \Psi^{m +\tilde{m}-2}_{sc}(M)$;
\end{itemize}
\end{proposition}
We have the following version of the sharp G{\aa}rding inequality for the above quantization of symbols in $S^m_{sc}(T^*M)$, which can be easily deduced from that adapted to symbols in $S^m_{sc}(\R^{2n})$ (see~\cite[Theorem~18.1.14]{Hoermander:V3} or~\cite[Theorem~2.5.4]{Lernerbook}) by using local charts.
\begin{theorem}
Let $m \in \R$ and assume that $a \in S^{2m+1}_{sc}(T^*M)$ is real valued and satisfies $a \geq 0$ on $T^*M$. Then, there exist $C, \lambda_0>0$ such that for all $u \in \Cinf(M)$ and all $\lambda \geq \lambda_0$ we have
$$
\Re(\Op(a)  u , u)_{L^2(M)} \geq -C \norm{ u }_{H^{m}_{sc}(M)}^2 .
$$
\end{theorem}
This last inequality can be equivalently rewritten as
$$
\Re(\Op(a)  u , u)_{L^2(M)} \geq -C \norm{u }_{H^m(M)}^2 -C \lambda^{2m}\norm{u }_{L^2(M)}^2.
$$
Note finally that the quantization formula defining $\Op(a) $ depends on the set of charts and of the partition of unity chosen. However, for $a \in S^m_{sc}(T^*M)$ one can show that its definition is intrinsic modulo $\Psi^{m-1}_{sc}(T^*M)$.
{
\subsection{Pseudodifferential operators on \texorpdfstring{$M = M' \times \T^{n''}$}{mboundary}}
The construction of pseudodifferential operators in the previous section is very general and does not take into account the particular product structure of the manifold $M=  M'\times \T^{n''}$. 
When taking into account the structure of $M =  M'\times \T^{n''}$ in the definition of the quantization $\Op(a)$, it is natural to choose product charts (and associated product partitions of unity): take $U_{kj} = W_k \times V_j$ where $\{W_k\}$ is an atlas of $\T^{n''}$ and $\{V_j\}$ an atlas of $M'$.
With such a choice of charts and partition of unity, we note that if $a_T \in \Cinf(T^* \T^{n''})$ and $a_{\Mp} \in \Cinf(T^* \Mp)$, then $a = a_T \otimes a_{\Mp}$ is quantified as
$$
\Op(a) = \Op(a_T) \Op(a_{\Mp}) = \Op(a_{\Mp}) \Op(a_T),
$$ 
where $\Op(a_T), \Op(a_{\Mp})$ denote the quantizations on $ \T^{n''}$ and $\Mp$ defined in local charts (relative to the charts $W_k$ and $V_j$ respectively, and associated partitions of unity). Similarly, it is convenient to take product charts and partitions of unity on $\T^{n''}$, i.e. take $\{W^\pm\}$ an atlas of $\T^1$ and choose each $W_k$ of the form $W^\pm \times\cdots \times W^\pm$. 
\subsection{Fourier multipliers on the torus \texorpdfstring{$\T^{n''}$}{s1}}
Given $u \in \Cinf(\T^{n''})$, we define by $\udl{u}$ is the unique $2\pi \Z^{n''}$-periodic function on $\R^{n''}$ coinciding with $u$ on $(0,2\pi]^{n''}$. We have in particular $\udl{u} \in \scrS'(\R^{n''})$. We denote by $\F(\udl{u})$ its Fourier Transform, where, for $v \in \scrS(\R^{n''})$ we use the following normalization of the Fourier transform:
$$
\F(v) (x)= \int v(x)e^{-i x \cdot \xi} dx , \quad \F^{-1}(v)(\xi) = \frac{1}{(2\pi)^{n''}}\int v(\xi)e^{i x \cdot \xi} d\xi  .
$$
Taking $\chi \in \Cinf(\R^{n''})$, we may define the Fourier multiplier $\chi(D)$ on $\Cinf(\T^{n''})$ by 
$$
(\chi(D) u)(x) = \F^{-1}\big(\chi(\xi) \F(\udl{u})(\xi) \big)(x) , \quad x \in \T^{n''} .
$$
Such Fourier multiplier are linked to the quantization $\Op(\chi)$ defined above on $\T^{n''}$ by the following proposition (for charts and partition of unity as in the previous section).
\begin{proposition}
\label{p:quantizqwtorus}
Take $\chi \in \Cinf(\R^{n''})$ with uniformly bounded derivatives. Then, we have
\begin{align*}
\chi(D/\lambda)  -  \Op(\chi(\xi/\lambda)) = O_{\calL(L^2(\T^{n''}))}(\lambda^{-1})
\end{align*}
\end{proposition}
This implies that 
\begin{equation}
\label{e:fourmult}
1 \otimes  \chi(D/\lambda) - \Op\big( 1 \otimes \chi(\xi/\lambda)\big) = O_{\calL(L^2(M))}(\lambda^{-1})
\end{equation}
Finally, this definition of Fourier multipliers is linked to Fourier series as follows. 
\begin{lemma}
\label{lem: pseudo circle}
Take $\chi \in \Cinf(\R^{n''})$ with uniformly bounded derivatives. Then, for any $u \in \Cinf(\T^{n''})$, we have
\begin{align*}
(\chi(D) u)(x) =  \sum_{k \in \Z^{n''}} \chi(k) \hat{u}_k e^{ik\cdot x} , \quad x \in \T^{n''} , \quad \hat{u}_k = \frac{1}{(2\pi)^{n''}} \int_{\T^{n''}} u(y) e^{-ik \cdot y} dy .
\end{align*}
 \end{lemma}
\begin{proof}
We first write $u(x) = \sum_{k \in \Z^{n''}} \hat{u}_k e^{ik \cdot x}$ on $\T^{n''}$, so that $\udl{u}(x) = \sum_{k \in \Z^{n''}} \hat{u}_k e^{ik \cdot x}$ on $\R$. Hence, we obtain
\begin{align*}
 \F(\udl{u})(\xi)  =\F(\sum_{k \in \Z^{n''}}  \hat{u}_k e^{ik\cdot x})(\xi) = \sum_{k \in \Z^{n''}} \hat{u}_k ((2\pi)^{n''}\delta_{\xi = k}) .
\end{align*}
This implies 
\begin{align*}
\F^{-1}\big(\chi(\xi) \F(\udl{u})(\xi) \big)(x)  & =\F^{-1}\big(\chi(\xi)  \sum_{k \in \Z^{n''}} \hat{u}_k (2\pi)^{n''}\delta_{\xi = k} \big)(x) 
=\F^{-1}\big(  \sum_{k \in \Z^{n''}} \hat{u}_k \chi(k) (2\pi)^{n''} \delta_{\xi = k} \big)(x) \\
& =\sum_{k \in \Z^{n''}} \chi(k) \hat{u}_k e^{ik \cdot x} ,
\end{align*}
which concludes the proof of the lemma.
\end{proof}
}
\section{Sharpness of Estimate~\texorpdfstring{\eqref{eq:estimate Rd}}{23}: converse of Theorem~
\texorpdfstring{\!\!\ref{th: estimate R^d}}{21}
}
In this appendix, we prove that our Estimate~\eqref{eq:estimate Rd} on $\R^d$ is optimal.
\begin{lemma}
\label{lem:converseestimRd}
Let $Q_{0}$ be given by \eqref{qzero} with $W$ satisfying \eqref{eq: hyp V}. Let us assume that there exists
a positive constant $\mu_{0}$ such that
$$
\forall \mu\ge \mu_{0},\exists \beta(\mu)>0, \forall u\in \Con^2_c(\R^d),\quad
 \|(Q_0 - \mu) u \|_{L^2(\R^d)} \geq \beta(\mu)
 \mu^{\frac{\gamma}{2\gamma + 1}} 
\|u \|_{L^2(\R^d)}.
$$
Then $\limsup_{\mu\rightarrow +\io} \beta(\mu)<+\io$.
\end{lemma}
\begin{proof}
 We consider the following affine symplectic mapping $(x,\xi)\mapsto (y,\eta)$
 of $\R^{2d}$:
 $$
\begin{cases}
 x=\mu^{\kappa} y,\\
 \xi_{1}=\mu^{-\kappa}\eta_{1}+\mu^{1/2},\quad \xi'=\mu^{-\kappa}\eta',
\end{cases}
 $$
 where $\kappa$ is a positive constant to be chosen later.
 The operator $Q_{0}-\mu$ has the Weyl symbol
 $
 \val \xi^{2}+iW(x)
 $
 and is thus unitarily equivalent to the operator $b^{w}$ with
 \begin{multline*}
b(y,\eta)=(\mu^{-\kappa}\eta_{1}+\mu^{1/2})^{2}+\mu^{-2\kappa}\val{\eta'}^{2}
 +W(\mu^{\kappa}y)-\mu
 \\=
 2\mu^{\frac12-\kappa}\eta_{1}
 +\mu^{-2\kappa}\val \eta^{2}+i\mu^{2\kappa\gamma}W(\mu^{\kappa}y)\mu^{-2\kappa\gamma}.
\end{multline*}
We choose now $\kappa$ so that $\frac12-\kappa=2\kappa\gamma$, \ie
$\kappa=1/(4\gamma+2)$ (entailing
$2\kappa\gamma=\gamma/(2\gamma+1)$)
and we obtain
$$
b(y,\eta)=\mu^{\frac{\gamma}{2\gamma+1}}c_{\mu}(y,\eta),\quad c_{\mu}(y,\eta)=\eta_{1}+i\underbrace{W(\mu^{\kappa}y)\mu^{-2\kappa\gamma}}_{\substack{\text{bounded from above} \\\text {and below by positive}\\\text{ constants}}
}
+{\val \eta}^{2}\mu^{-\kappa-\frac12}.
$$
Let $u\in \Con^\io_c(\R^d)$ with $L^{2}$ norm 1. From our assumption in the lemma, we find with some unitary $U_{\mu}$
$$
\norm{U_{\mu}^{*}c_{\mu}^{w} U_{\mu}u}_{L^{2}(\R^{d})}\ge \beta(\mu)\norm{u}_{L^{2}(\R^{d})}\Longrightarrow
\norm{c_{\mu}^{w} U_{\mu}u}_{L^{2}(\R^{d})}\ge \beta(\mu).
$$
Since the mapping $U_{\mu}$ is also an isomorphism of $\Con^\io_c(\R^d)$
(from the particular form of the symplectic mapping,  the mapping $U_{\mu}$ is the composition of a multiplication by a factor $e^{i\alpha x_{1}}$ with a rescaling
$v\mapsto v(\lambda x) \lambda^{d/2}$ ), we may choose
$$
u=U_{\mu}^{*}w_{0},\quad \text{$w_{0}\in \Con^\io_c(\R^d)$ with $L^{2}$ norm 1,}
$$
and we obtain
$\norm{c_{\mu}^{w} w_{0}}_{L^{2}(\R^{d})}\ge \beta(\mu)$, which implies 
\begin{multline*}
\limsup_{\mu\rightarrow+\io} \beta(\mu)
\le \limsup_{\mu\rightarrow+\io} \norm{c_{\mu}^{w} w_{0}}_{L^{2}(\R^{d})}\\\le \norm{D_{1}w_{0}}_{L^{2}(\R^{d})}
+C_{1} \norm{w_{0}}_{L^{2}(\R^{d})}+\limsup_{\mu\rightarrow+\io}\mu^{-\kappa-\frac12} \norm{\Delta w_{0}}_{L^{2}(\R^{d})}\\ =\norm{D_{1}w_{0}}_{L^{2}(\R^{d})}
+C_{1} \norm{w_{0}}_{L^{2}(\R^{d})}<+\io,
\end{multline*}
which gives the result.
\end{proof}
\small
\bibliographystyle{alpha}     
\bibliography{bibli}
\end{document}